\documentclass[a4paper,11pt,twoside]{amsart}

\usepackage[latin1]{inputenc}
\usepackage{amsmath}
\usepackage{amsthm}
\usepackage{amssymb,amscd}
\usepackage[all]{xy}
\usepackage{hyperref}

\newtheorem{thm}{Theorem}[section]
\newtheorem{prop}[thm]{Proposition}
\newtheorem{lem}[thm]{Lemma}
\newtheorem{cor}[thm]{Corollary}
\newtheorem{conjecture}[thm]{Conjecture}

\numberwithin{equation}{section}

\theoremstyle{definition}
\newtheorem{definition}[thm]{Definition}
\newtheorem{remark}[thm]{Remark}
\newtheorem{ex}[thm]{Example}

\def\cont{{\, \subseteq \, }}
\newcommand{\os}[2]{\overset{#1}{#2}}

\newcommand{\ol}[1]{\overline {#1}}

\newcommand{\grf}[1]{\mbox{$\left \{ #1 \right \}$}}

\newcommand{\mc}[1]{\mathcal{#1}}
\def\glpiu{\widetilde{\rm Gl}_2^+(\mathbb{R})}
\def\glp{{\rm Gl}_2^+(\mathbb{R})}
\def\aa{\alpha}

\def\ff{\phi}

\def\pp{\pi}

\def\ss{\sigma}
\def\tt{\tau}

\def\ww{\omega}

\def\Ff{\Phi}
\def\Gg{\Gamma}

\def\Ss{\Sigma}

\newcommand{\HH}{\mathbb{H}}
\newcommand{\Db}{{\rm D}^{\rm b}}

\newcommand{\Km}{{\rm Km}}

\newcommand{\Aut}{{\rm Aut}}

\newcommand{\Br}{{\rm Br}}
\newcommand{\NS}{{\rm NS}}
\newcommand{\Pic}{{\rm Pic}}

\newcommand{\rk}{{\rm rk}\,}
\newcommand{\coh}{{\cat{Coh}}}

\newcommand{\End}{{\rm End}}
\newcommand{\Hom}{{\rm Hom}}

\newcommand{\Rig}{{\rm Rig}}
\newcommand{\Sph}{{\rm Sph}}
\newcommand{\OO}{{\rm O}}
\newcommand{\nk}{{\rm ker}}
\newcommand{\co}{{\rm coker}}
\newcommand{\Inf}{{\rm Inf}}
\newcommand{\Forg}{{\rm Forg}}
\newcommand{\NStab}{{\rm Stab}_{\kn}}
\newcommand{\NDStab}{{\rm Stab}^\dag_{\kn}}

\newcommand{\lto}{\longrightarrow}
\newcommand{\Stab}{{\rm Stab}}

\newcommand{\id}{{\rm id}}

\newcommand{\dual}{^{\vee}}
\newcommand{\mono}{\hookrightarrow}

\newcommand{\mor}[1][]{\xrightarrow{#1}}
\newcommand{\isomor}{\mor[\sim]}

\newcommand{\cat}[1]{\begin{bf}#1\end{bf}}

\newcommand{\KE}{{\rm Ker}^{G_\Delta}}

\newcommand{\cal}{\mathcal}
\newcommand{\ka}{{\cal A}}
\newcommand{\kb}{{\cal B}}
\newcommand{\kc}{{\cal C}}

\newcommand{\ke}{{\cal E}}
\newcommand{\kf}{{\cal F}}
\newcommand{\kg}{{\cal G}}
\newcommand{\kh}{{\cal H}}

\newcommand{\kn}{{\cal N}}

\newcommand{\ko}{{\cal O}}
\newcommand{\kp}{{\cal P}}

\newcommand{\ks}{{\cal S}}
\newcommand{\kt}{{\cal T}}

\newcommand{\kz}{{\cal Z}}

\newcommand{\ZZ}{\mathbb{Z}}
\newcommand{\QQ}{\mathbb{Q}}
\newcommand{\RR}{\mathbb{R}}
\newcommand{\CC}{\mathbb{C}}
\newcommand{\PP}{\mathbb{P}}

\newcommand{\til}[1]{\widetilde{#1}}

\begin{document}

\title{Inducing stability conditions}

\author[E.\ Macr\`\i, S.\ Mehrotra, and P.\ Stellari]{Emanuele Macr\`\i, Sukhendu Mehrotra, and Paolo Stellari}


\address{E.M.: Department of Mathematics, University of Utah, 155 South 1400 East, Salt
Lake City, UT 84112-0090, USA} \email{macri@math.utah.edu}

\address{S.M.: Department of Mathematics and Statistics, University of Massachusetts--Amherst, 710 N. Pleasant Street, Amherst,
MA 010003, USA}\email{mehrotra@math.umass.edu}

\address{P.S.: Dipartimento di Matematica ``F. Enriques'',
Universit{\`a} degli Studi di Milano, Via Cesare Saldini 50, 20133
Milano, Italy} \email{paolo.stellari@unimi.it}



\begin{abstract} We study stability conditions induced by functors
between triangulated categories. Given a finite group acting on a
smooth projective variety we prove that the subset of invariant
stability conditions embeds as a closed submanifold into the
stability manifold of the equivariant derived category. As an
application we examine stability conditions on Kummer and Enriques
surfaces and we improve the derived version of the Torelli Theorem
for the latter surfaces already present in the litterature. We
also study the relationship between stability conditions on
projective spaces and those on their canonical
bundles.\end{abstract}

\maketitle

\section{Introduction}\label{sec:intro}

Stability conditions on triangulated categories were introduced by Bridgeland in \cite{B1}
following work of Douglas {\cite{Do}) on $\Pi$-stability in string theory. A central feature
of this construction is that, under mild conditions, the set of stability conditions on a
triangulated category carries a natural structure of a complex manifold. Since then, a fair
amount of effort has gone into studying such spaces of stability conditions for derived categories
of coherent sheaves on low-dimensional projective manifolds, due, in large part, to the intuition
that these stability manifolds should be (approximate) mathematical models of the physicists'
stringy K\"{a}hler moduli space (see \cite{BR} and \cite{Do}), and also because of the
close connection their geometry is expected to have with properties of the group of autoequivalences
of the derived category (\cite{B2, HMS}). Stability manifolds for curves have been explicitly
described in \cite{B1,Okada1,emolo}, while for minimal
resolutions of type-$A$ surface singularities, a complete
topological description of these spaces has been obtained in
\cite{IUU} (see also \cite{Okada2}). In particular, the stability manifold is connected and simply
connected in these cases.
The situation is rather more intricate for algebraic K3 surfaces, and
neither the connectedness nor the simply connectedness has been
proved. Nevertheless one distinguished connected component has
been identified and related to the problem of describing the group
of autoequivalences of the derived category of the surface
(\cite{B2}). For smooth generic analytic K3 surfaces a complete
picture is given in \cite{HMS}.

A basic difficulty in the theory, however, remains: how to systematically
construct examples of stability conditions, at least when the
geometry of the variety is well-understood? Naively, if two
varieties $X$ and $Y$  are related in some intimate geometric way,
one would expect to be able to solve the previous problem for $X$ once it has been
solved for $Y$. In other words the geometric connection between $X$
and $Y$ should produce some, perhaps weak, relation between their
stability manifolds.

In this paper, we develop a technique of inducing stability conditions via functors between triangulated categories with nice properties and show how, in certain geometric contexts, this procedure gives an answer to the problem mentioned above. In \cite{Pol}, Polishchuk proposed similar ideas from a somewhat different perspective (see the end of Section \ref{subsec:stabex}).

The first situation where this approach will be applied is the case of a smooth projective variety $X$ with the action of a finite group $G$. Forgetting linearizations on equivariant complexes gives a natural faithful functor between the equivariant derived category $\Db_G(X)$ and $\Db(X)$. Since $G$ acts in a natural way on the stability manifold $\Stab(\Db(X))$ of $\Db(X)$, one can consider the subset of invariant stability conditions on the derived category $\Db(X)$. Our first result is now the following:

\begin{thm}\label{thm:main2} The subset $\Gamma_X$ of invariant
stability conditions in $\Stab(\Db(X))$ is a closed submanifold.
The forgetful functor $\Forg_G$ induces a closed embedding
$$\Forg_G^{-1}:\Gamma_X\hookrightarrow\Stab(\Db_G(X))$$ such that
the semistable objects in $\Forg_G^{-1}(\sigma)$ are the objects
$\ke$ in $\Db_G(X)$ such that $\Forg_G(\ke)$ is semistable in
$\sigma\in\Gamma_X$.\end{thm}

As we will point out, an analogous statement holds when we restrict to numerical stability conditions. Note that the existence of a bijection between the subset of invariant stability conditions in $\Stab(\Db(X))$ and a certain subset of $\Stab(\Db_G(X))$ was already observed in \cite{Pol}.

A possible application of this theorem could be in the
construction of stability conditions on projective Calabi-Yau
threefolds. Let us briefly outline the strategy. Often, the
equivariant derived category of coherent sheaves of a variety with
a finite automorphism is a category generated by a strong
exceptional collection, (see \cite{GL}, or \ref{ex:projline}).
Stability conditions on such categories are relatively easy to
manufacture. Starting thus with a suitable Calabi-Yau threefold
$X$ with a finite automorphism, one would want to construct a
stability condition on the equivariant category which can be
deformed into the image under $\Forg_G^{-1}$ of the stability
manifold of $X$. Retracing one's steps through the
forgetful functor, one would then obtain a stability condition on
$X$ itself. Stability conditions on Calabi-Yau threefolds have seen
considerable interest recently thanks to the work of Kontsevich
and Soibelman on counting invariants. In the announced preprint
\cite{KontSoib}, the authors develop a generalized theory of
Donaldson-Thomas invariants for 3-Calabi-Yau categories which obey
certain wall-crossing formulas on the space of stability
conditions. The reader is encouraged to consult \cite{Joyce4, PT1,
PT2, Arend, Toda5, Toda6} for motivation and applications.

Let us briefly mention the two easy examples we consider in Section \ref{subsec:exinvar}
 to illustrate Theorem \ref{thm:main2}.
In \cite{GL}, the authors study the derived categories of certain weighted projective lines which are in fact stacks obtained as quotients of some plane cubics $E$ by the action of a natural involution. Our result then realizes the stability manifold of $E$ as a closed submanifold of the stability manifold of the derived category of each of these (stacky) weighted projective lines (see Example \ref{ex:projline}).

More interesting is the example of Kummer surfaces (see Example \ref{ex:Kummer}). In \cite{B2} Bridgeland describes the connected component of maximal dimension of the space $\NStab(\Db(A))$ parametrizing numerical stability conditions on an abelian surface $A$. Using the equivalence $\Db(\Km(A))\cong\Db_G(A)$ (see \cite{BKR}), where $G$ is the group generated by the natural involution on $A$, we show that this is embedded as a closed submanifold into a distinguished component of the space $\NStab(\Db(\Km(A)))$ of numerical stability conditions on the Kummer surface $\Km(A)$. The latter component was also studied in \cite{B2} and its topology is related to the description of the group of autoequivalences of $\Db(\Km(A))$.

This general philosophical approach to inducing stability conditions between close geometric
relatives is clearly reflected in our treatment of Enriques surfaces, which form the main example
to which we apply the techniques of Section \ref{sec:inducing}. Recall that an Enriques surface $Y$ is
a minimal smooth projective surface with $2$-torsion canonical bundle $\omega_Y$ and $H^1(Y,\ko_Y)=0$. The universal cover $\pi:X\to Y$ is a K3 surface and it carries a fixed-point-free involution $\iota:X\to X$ such that $Y=X/G$, where $G=\langle\iota\rangle$.

In \cite{B2}, Bridgeland studied  a distinguished connected
component $\NStab^\dag(\Db(X))$ of the stability manifold
$\NStab(\Db(X))$ parametrizing numerical stability conditions on
$\Db(X)$. The idea here is to relate this component to some
interesting connected component of $\NStab(\Db(Y))$. Write
$\OO(\widetilde H(X,\ZZ))_G$ for the group of $G$-equivariant
Hodge isometries of the Mukai lattice of $X$ and denote by
$\OO_+(\widetilde H(X,\ZZ))_G\subset\OO(\widetilde H(X,\ZZ))_G$
the index-$2$ subgroup of $G$-equivariant orientation preserving
Hodge isometries. Then our second main result gives a precise
description of this relationship:

\begin{thm}\label{thm:main1} Let $Y$ be an Enriques surface and $\pi:X\to Y=X/G$ its
universal cover.
\begin{itemize}
\item[{\rm (i)}] There exists a connected component
$\NStab^\dag(\Db(Y))$ of $\NStab(\Db(Y))$ naturally embedded into
$\NStab(\Db(X))$ as a closed submanifold. Moreover, if $Y$ is generic, the category $\Db(Y)$ does not contain spherical objects and $\NStab^\dag(\Db(Y))$ is isomorphic to the distinguished connected component $\NStab^\dag(\Db(X))$.
\item[{\rm (ii)}] There is a natural homorphism of groups $$\Aut(\Db(Y))\lto\OO(\widetilde H(X,\ZZ))_G/G$$ whose image is the index-2 subgroup $\OO_+(\widetilde H(X,\ZZ))_G/G$.
\end{itemize}
\end{thm}

To avoid confusion, it is perhaps worth pointing out that the
embedding in (i) is not realized by the forgetful functor as in
Theorem \ref{thm:main2} but by an adjoint of this functor.
Nevertheless, Theorem \ref{thm:main2} will be used in the proof of
Theorem \ref{thm:main1} to identify the connected component
$\NStab^\dag(\Db(Y))$. As we will see in Proposition
\ref{prop:rigid}, in the generic case one can be even more
explicit about the classification of (strongly) rigid objects in
$\Db(Y)$.

As explained in Section \ref{subsec:autoeq}, part (ii) of the
previous result can be seen as an improvement of the derived
version of the Torelli Theorem for Enriques surfaces proved in
\cite{BM}. This result asserts that, given two Enriques surfaces
$Y_1$ and $Y_2$ with universal covers $X_1$ and $X_2$,
$\Db(Y_1)\cong\Db(Y_2)$ if and only if there exists a Hodge
isometry $\widetilde H(X_1,\ZZ)\cong\widetilde H(X_2,\ZZ)$ of the
total cohomology groups which is equivariant with respect to the
involutions $\iota_1$ and $\iota_2$ defined on $X_1$ and $X_2$. As
a consequence of our results (see Corollary \ref{cor:dtt}) we get
a characterization of the Hodge isometries $\widetilde
H(X_1,\ZZ)\cong\widetilde H^2(X_2,\ZZ)$ induced by all possible
Fourier--Mukai equivalences $\Db(Y_1)\cong\Db(Y_2)$.

Furthermore, the relation between the connected component
$\NStab^\dag(\Db(Y))$ and the description of $\Aut(\Db(Y))$ is
quite deep. Indeed we prove that there exists a covering map from
an open and closed subset $\Sigma(Y)\subseteq\NStab(\Db(Y))$
containing $\NStab^\dag(\Db(Y))$, onto some period domain and a
subgroup of $\Aut(\Db(Y))$ acts as the group of deck
transformations. As it will be explained, we expect
$\Sigma(Y)=\NStab^\dag(\Db(Y))$ and the geometric picture is very
similar to the one given by Bridgeland for K3 surfaces. In
particular we state a conjecture (Conjecture \ref{conj:BriEnr})
about $\Aut(\Db(Y))$ which extends Conjecture 1.2 in \cite{B2} to
the case of Enriques surfaces.

Finally we study two other geometric situations for which the procedure of inducing stability conditions via faithful functors can be exploited. First of all, we compare stability conditions on projective spaces with those on their canonical bundles. Our result in this direction is Theorem \ref{thm:stab1} whose main part can be summarized as follows:

\begin{thm}\label{thm:main3} An open subset of $\Stab(\Db(\PP^1))$
embeds into the stability manifold of the total space of the
canonical bundle $\ww_{\PP^1}$ as a fundamental domain for the
action of the group of autoequivalences.\end{thm}

It is known by the work in \cite{Okada2,IUU} that the space of stability conditions on the canonical bundle of the projective line is connected and simply connected.  We provide a new simpler proof of these topological properties based on \cite{HMS}, on the way to proving the previous result.

The last example is concerned with the relationship between the spaces of stability conditions on resolutions of Kleinian singularities and those of the corresponding quivers, with particular attention to the $A_2$-singularity case. The picture we get is completely similar to the one presented in the previous theorem.

\medskip

The plan of the paper is as follows. In Section \ref{sec:inducing} we recall Bridgeland's construction of stability conditions and the metric properties of the stability manifold. We then show how a faithful functor induces stability conditions and apply these results to equivariant derived categories.

In Section \ref{sec:equiv} we study Enriques surfaces and show
that a connected component of the stability manifold of such a
surface embeds as a closed submanifold into the space of stability
conditions of its universal cover. An improvement of the derived
Torelli Theorem for Enriques surfaces is proved and we
conjecturally relate the topology of the distinguished connected
component to the group of autoequivalences. Some special
properties of generic Enriques surfaces are then described.

Finally, Section \ref{sec:CY} treats the case of projective spaces and of the corresponding cotangent bundles and further examples.

Unless otherwise stated all varieties considered in this paper are meant to be smooth projective and defined over the complex numbers.

\section{Exact functors and invariant stability conditions}\label{sec:inducing}

In this section we show how to induce stability conditions using
exact functors with nice properties. This permits us to relate the
stability manifolds of different projective varieties taking into
account certain geometric relations between the varieties
themselves. Notice that a criterion for the existence of induced
stability conditions in geometric contexts (including the ones we
will consider) can be easily deduced from \cite{Pol} (see Theorem
\ref{thm:Poli}).

As a first application, in Section \ref{subsec:gen} we prove
Theorem \ref{thm:main2} according to which, given a smooth
projective variety with an action of a finite group, the closed
submanifold of invariant stability conditions embeds into the
stability manifold of the equivariant derived category. We use
this to show that the unique connected component of maximal
dimension of the stability manifold of an abelian surface embeds
into Bridgeland's connected component of the stability manifold of
the associated Kummer surface.

\bigskip

\subsection{Bridgeland's framework}\label{subsec:bri}~

\medskip

In this section we recall a few results from \cite{B1} which will
be used throughout this paper. For the moment, let $\cat{T}$ be an
essentially small triangulated category and let $K(\cat{T})$ be
its Grothendieck group.

\begin{definition}\label{def:main} A \emph{stability condition} on $\cat{T}$
is a pair $\sigma=(Z,\kp)$ where $Z:K(\cat{T})\to\CC$ is a group
homomorphism (the \emph{central charge}) and
$\kp(\phi)\subset\cat{T}$ are full additive subcategories,
$\phi\in\RR$, satisfying the following conditions:

(a) If $0\ne E\in\kp(\phi)$, then $Z(E)=m(E)\exp(i\pi\phi)$ for
some $m(E)\in\RR_{>0}$.

(b) $\kp(\phi+1)=\kp(\phi)[1]$ for all $\phi\in\RR$.

(c) If $\phi_1>\phi_2$ and $E_i\in\kp(\phi_i)$, $i=1,2$, then
$\Hom_{\cat{T}}(E_1,E_2)=0$.

(d) Any $0\ne E\in\cat{T}$ admits a \emph{Harder--Narasimhan
filtration} (\emph{HN-filtration} for short) given by a collection
of distinguished triangles $E_{i-1}\to E_i\to A_i$ with $E_0=0$
and $E_n=E$ such that $A_i\in\kp(\phi_i)$ with
$\phi_1>\ldots>\phi_n$.\end{definition}

It can be shown that each subcategory $\mc{P} (\phi)$ is
extension-closed and abelian. Its non-zero objects are called
\emph{semistable} of phase $\phi$, while the objects $A_i$ in (d)
are the \emph{semistable factors} of $E$. The minimal objects of
$\kp(\phi)$ are called \emph{stable} of phase $\phi$ (recall that
a \emph{minimal object} in an abelian category, also called
\emph{simple}, is a non-zero object without proper subobjects or
quotients). A HN-filtration of an object $E$ is unique up to a
unique isomorphism. We write $\phi^+_\sigma (E) := \phi_1$,
$\phi^-_\sigma (E) := \phi_n$, and $m_\sigma (E) := \sum_j
|Z(A_j)|$.

For any interval $I \subseteq \RR$, $\mc{P} (I)$ is defined to be
the extension-closed subcategory of $\cat{T}$ generated by the
subcategories $\mc{P} (\phi)$, for $\phi \in I$. Bridgeland proved
that, for all $\phi \in \RR$, $\mc{P} ((\phi, \phi + 1])$ is the
heart of a bounded $t$-structure on $\cat{T}$. The category
$\mc{P} ((0, 1])$ is called the \emph{heart} of $\sigma$.
In general, the category $\kp((a,b))\subseteq\cat{T}$, for $a,b\in\RR$ with $0<b-a\leq1$, is \emph{quasi-abelian} (see \cite[Sect.\ 4]{B1}) and the strict short exact sequences are the distinguished triangles in $\cat{T}$ whose vertices are all in $\kp((a,b))$.

\begin{remark}\label{rmk:tstruct} As pointed out in \cite[Prop.\ 5.3]{B1} to exhibit a stability condition on a triangulated category $\cat{T}$, it is enough to give a bounded $t$-structure on $\cat{T}$ with heart $\cat{A}$ and a group homomorphism $Z:K(\cat{A})\to\CC$ such that $Z(E)\in\HH$, for all $0\neq E\in\cat{A}$, and with the Harder--Narasimhan property (see \cite[Sect.\ 5.2]{B1}). Recall that $\HH:=\{ z\in\CC^*:z=|z|\exp(i\pi\phi), \, 0< \phi \leq 1 \}$ and that the above homomorphism $Z$ is called a \emph{stability function}.

As a special case, if $\cat{A} \subseteq \cat{T}$ is the heart of
a bounded $t$-structure and moreover it is an abelian category of
finite length (i.e.\ artinian and noetherian), then a group
homomorphism $Z: K(\cat{A}) \to \CC$  with  $Z(E)\in\HH$, for all
minimal objects $E\in\cat{A}$, extends to a unique stability
condition on $\cat{T}$.\end{remark}

A stability condition is called \emph{locally-finite} if there
exists some $\epsilon > 0$ such that, for all $\phi\in\RR$, each
quasi-abelian subcategory  $\mc{P} ((\phi - \epsilon , \phi +
\epsilon))$ is of finite length. In this case $\mc{P} (\phi)$ has
finite length so that every object in $\mc{P} (\phi)$ has a finite
\emph{Jordan--H\"older filtration} (\emph{JH-filtration} for
short) into stable factors of the same phase. The set of stability
conditions which are locally finite will be denoted by $\Stab
(\cat{T})$.

By \cite[Prop.\ 8.1]{B1} there is a natural topology on $\Stab
(\cat{T})$ defined by the generalized metric
\begin{equation}\label{eq:metric}
d(\sigma_1, \sigma_2) := \underset{0 \neq E \in \cat{T}}{\sup}
\left\lbrace |\phi_{\sigma_2}^+ (E) - \phi_{\sigma_1}^+ (E)|,
|\phi_{\sigma_2}^- (E) - \phi_{\sigma_1}^- (E)|, \left\vert\log
\frac{m_{\sigma_2} (E)}{m_{\sigma_1} (E)}\right\vert \right\rbrace
\in [0, \infty].
\end{equation}

\begin{remark}\label{rmk:action} Bridgeland proved in \cite[Lemma 8.2]{B1} that there are two groups which naturally act on $\Stab(\cat{T})$. The first one is the group of exact autoequivalences $\Aut(\cat{T})$ which, moreover, preserves the structure of generalized metric space just defined.

The universal cover $\glpiu$ of $\glp$ acts on the right in the
following way. Let $(G,f) \in \glpiu$, with $G \in\glp$ and $f:
\RR \to \RR$ an increasing map such that $f(\phi+1)=f(\phi)+1$ and
$G\exp(i\pi \phi)/ | G\exp(i\pi \phi)| =\exp(2i \pi f(\phi))$, for
all $\phi \in \RR$. Then $(G, f)$ maps $(Z, \kp) \in \Stab
(\cat{T})$ to $(G^{-1} \circ Z, \kp \circ f)$.\end{remark}

The result from \cite{B1} that we will need for the rest of the
paper is the following:

\begin{thm}\label{thm:bridmain}{\bf (\cite{B1}, Theorem 1.2.)}
For each connected component $\Sigma \subseteq \Stab (\cat{T})$
there is a linear subspace $V(\Sigma) \subseteq (K(\cat{T})
\otimes \CC)^{\vee}$ with a well-defined linear topology such that
the natural map $$\mc{Z}:\Sigma\lto V(\Sigma),\;\;\;\;\;\;\;(Z,
\mc{P})\longmapsto Z$$ is a local homeomorphism. In particular, if
$K(\cat{T}) \otimes \CC$ is finite dimensional, $\Sigma$ is a
finite dimensional complex manifold.
\end{thm}

The complex manifold $\Stab (\cat{T})$ will be called the
\emph{stability manifold} of $\cat{T}$.


Suppose now that the category $\cat{T}$ is $\CC$-linear and of
finite type. The Euler--Poincar\'e form on $K(\cat{T})$
\[
\chi(E,F):=\sum_{i\in\ZZ}(-1)^i(E,F)^i,
\]
where $(E,F)^i:=\dim_\CC\Hom_{\cat{T}}(E,F[i])$, allows us to
define the \emph{numerical Grothendieck group}
$\kn(\cat{T})=K(\cat{T})/K(\cat{T})^\perp$ (the orthogonal is with
respect to $\chi$). We will say that $\cat{T}$ is
\emph{numerically finite} if the rank of $\kn(\cat{T})$ is finite.
To shorten notation, when $\cat{T}=\Db(X):=\Db(\coh(X))$, for $X$
smooth and projective variety over $\CC$, we will write $\kn(X)$
instead of $\kn(\Db(X))$. Notice that, in such a case, by the
Riemann--Roch Theorem guarantees that $\Db(X)$ is numerically
finite. For $X$ a K3 surface, then
$\kn(X)=H^0(X,\ZZ)\oplus\NS(X)\oplus H^4(X,\ZZ)$, where $\NS(X)$
denotes the N\'eron--Severi group of $X$.

Assume that $\cat{T}$ is numerically finite. A stability condition
$\sigma=(Z,\kp)$ such that $Z$ factors through
$K(\cat{T})\twoheadrightarrow\kn(\cat{T})$ is called
\emph{numerical}. We denote by $\NStab(\cat{T})$ the complex
manifold parametrizing numerical stability conditions. As an
immediate consequence of the definition, an analogue of Theorem
\ref{thm:bridmain} holds true in the numerical setting (see
\cite[Cor.\ 1.3]{B1}).

We conclude this section with a discussion of two examples of
stability conditions needed in the sequel.

\begin{ex}\label{ex:K3ab} {\bf (K3 and abelian surfaces)} We briefly recall
the construction in \cite{B2} for abelian or K3 surfaces $X$.
Fix
$\omega,\beta\in\NS(X)\otimes\QQ$ with $\omega$ in the ample cone
and define the categories $\kt(\omega,\beta)$ consisting of
sheaves whose torsion-free part have $\mu_\omega$-semistable
Harder--Narasimhan factors with slope greater than
$\omega\cdot\beta$ and $\kf(\omega,\beta)$ consisting of
torsion-free sheaves whose $\mu_\omega$-semistable
Harder--Narasimhan factors have slope smaller or equal to
$\omega\cdot\beta$. Next consider the abelian category
\[
\ka(\omega,\beta):=\left\{\ke\in\Db(X):\begin{array}{l}
\bullet\;\;\kh^i(\ke)=0\mbox{ for }i\not\in\{-1,0\},\\\bullet\;\;
\kh^{-1}(\ke)\in\kf(\omega,\beta),\\\bullet\;\;\kh^0(\ke)\in\kt(\omega,\beta)\end{array}\right\}
\]
and the $\CC$-linear map
\[
Z_{\omega,\beta}:\kn(X)\lto\CC,\;\;\;\;\;\;\;\;\;\ke\longmapsto\langle\exp{(\beta+i\omega)},v(\ke)\rangle,
\]
where $v(\ke)$ is the Mukai vector of $\ke\in\Db(X)$ and
$\langle-,-\rangle$ is the Mukai pairing (see \cite[Ch.\
10]{Hu1}). By \cite[Lemma 6.2, Prop.\ 7.1]{B2}, if
$\omega\cdot\omega>2$, the pair
$(Z_{\omega,\beta},\ka(\omega,\beta))$ defines a stability
condition (Remark \ref{rmk:tstruct}).

In the rest of the paper we will be particularly interested in the
connected component $$\NDStab(\Db(X))\subseteq\NStab(\Db(X))$$
extensively studied in \cite{B2}. It can be described as the
connected component containing the stability conditions
$(Z_{\omega,\beta},\ka(\omega,\beta))$ with $\omega$ and $\beta$
as above.\end{ex}

\begin{ex}\label{ex:PN} {\bf (Projective spaces)} Recall that an object $E$ in a triangulated category $\cat{T}$ is \emph{exceptional} if
\[
\Hom^i_{\cat{T}}(E,E)\cong\left\lbrace\begin{array}{ll} \CC &
\mbox{if } i=0\\0 & \mbox{otherwise}\end{array}\right.
\]
An ordered collection of exceptional objects
$E=\grf{E_0,\ldots,E_n}$ is \emph{strong exceptional} in $\cat{T}$
if $\Hom^k_{\cat{T}}(E_i,E_j) \neq 0$ only if $i\leq j$ and $k=0$.
A strong exceptional collection of two objects is a \emph{strong
exceptional pair}. Finally, a strong exceptional collection is
\emph{complete}, if $E$ generates $\cat{T}$ by shifts and
extensions.

By \cite{Beil, GR} we know that $\Db (\PP^N)$ admits a complete
strong exceptional collection given by $\grf{\mc{O}, \ldots ,
\mc{O}(N)}$. Fix $E=\grf{\ke_0, \ldots,\ke_N}$ to be a strong
complete exceptional collection on $\Db (\PP^N)$. We construct
some explicit examples of stability conditions associated to $E$
in the following way (we use freely \cite{emolo}).

By \cite[Lemma 3.14]{emolo}, the subcategory
$\langle E \rangle_{\ol{p}} := \langle\ke_0 [p_0], \ldots,\ke_N
[p_N] \rangle\subseteq\cat{T}$ generated by extensions by $\ke_0 [p_0], \ldots,\ke_N
[p_N]$, for a collection of integers $\ol{p}=\{
p_0,\ldots,p_N\}$, with $p_0 > p_1 > \ldots > p_N$, is the heart
of a bounded $t$-structure on $\Db (\PP^N)$. Then $\langle E
\rangle_{\ol{p}}$ is an abelian category of finite length and the
Grothendieck group $K(\PP^N)$ is a free abelian group of finite
rank isomorphic to $\ZZ^{\oplus (N+1)}$ generated by the classes
of $\ke_0,\ldots,\ke_N$. Fix $z_0,\ldots, z_N\in\HH$ and define a
stability function
\[
Z_{\ol{p}} : K(\langle E \rangle_{\ol{p}}) \lto
\CC,\;\;\;\;\;\;\;\;\;\ke_i [p_i]\longmapsto z_i,
\]
for all $i$. By Remark \ref{rmk:tstruct} this extends to a unique
locally finite stability condition $\sigma_{\ol{p}}^E$ on
$\Db(\PP^N)$.

Define $\Theta_{E}$ as the subset of $\Stab (\PP^N) := \Stab (\Db
(\PP^N)) = \NStab (\Db (\PP^N))$ consisting of stability
conditions $\sigma$ of the form $\sigma=\sigma_{\ol{p}}^E\cdot(G,f)$,
for some strictly decreasing collection of integers
$\ol{p}=\{p_0,\ldots,p_N\}$ and for $(G,f)\in\glpiu$. By
\cite[Lemma 3.16]{emolo}, $\ke_0,\ldots,\ke_N$ are stable in all
stability conditions in $\Theta_{E}$. Moreover, for a stability
condition $\sigma_{\ol{p}}^E$ having $\rk_{\RR} Z_{\ol{p}} =1$
(thinking of $Z_{\ol{p}}$ as a map from $K (\PP^N)\otimes\RR$ to
$\CC\cong\RR^2$), $\ke_0,\ldots,\ke_N$ are the only stable objects
(up to shifts). Lemma 3.19 of \cite{emolo} shows that
$\Theta_{E}\cont\Stab(\Db(\PP^N))$ is an open, connected and
simply connected $(N+1)$-dimensional submanifold.\end{ex}

\bigskip

\subsection{Construction of induced stability conditions}\label{subsec:stabex}~

\medskip

In general stability, conditions do not behave well with respect
to exact functors between triangulated categories. What we are
going to show is that in the particular cases discussed below, it
may be possible to induce stability conditions from one category
to another (see \cite{Pol} and also \cite{emolo2,Toda2}).

Let $F : \cat{T} \to \cat{T}'$ be an exact functor between two
essentially small triangulated categories. Assume that $F$
satisfies the following condition:
\begin{equation*}\tag{\rm Ind}\label{eq:ind}
\Hom_{\cat{T}'} (F(A), F(B)) = 0 \ \mbox{ implies }\
\Hom_{\cat{T}} (A, B) = 0, \qquad \mbox{for any } A, B \in
\cat{T}.
\end{equation*}
For example, if $F$ is faithful, condition \eqref{eq:ind} holds.
Notice that, in particular, if \eqref{eq:ind} holds, then
$F(A)\cong 0$, for some $A\in\cat{T}$, implies that $A\cong 0$.
Let $\sigma'= (Z', \mc{P}') \in \Stab (\cat{T}')$ and define
$\sigma = F^{-1} \sigma' = (Z, \mc{P})$ by
\begin{align*}
\ & Z = Z' \circ F_*, \\
\ & \mc{P} (\phi) = \{ E \in \cat{T} \, : \, F(E)\in \mc{P'}
(\phi) \},
\end{align*}
where $F_*:K(\cat{T})\otimes\CC\to K(\cat{T'})\otimes\CC$ is the
natural morphism induced by $F$.


\begin{remark}\label{rmk:HN} (i) The categories $\mc{P} (\phi)$ are additive and extension-closed. Moreover $\sigma$ satisfies the first three properties of Definition \ref{def:main}.
Hence, in order to prove that $\sigma$ is a stability condition on
$\cat{T}$, it will be sufficient to prove that HN-filtrations
exist.

(ii) Once we know that HN-filtrations exist in $\sigma$, then
local-finiteness is automatic. Indeed $F$ induces a functor
$\mc{P} ((\phi-\epsilon,\phi +\epsilon))\to\mc{P}' ((\phi -
\epsilon , \phi + \epsilon))$ which, by definition, maps strict
short exact sequences into strict short exact sequences. Now
condition \eqref{eq:ind} guarantees that if we have a strict
inclusion $A\mor[l]B$, with $A,B\in\mc{P} ((\phi-\epsilon,\phi
+\epsilon))$, such that the induced map $F(A)\mor[F(l)]F(B)$ is an
isomorphism, then also $A\mor[l]B$ is an isomorphism. Hence, an
easy check shows that $\kp((\phi-\epsilon,\phi +\epsilon))$ is of
finite-length, provided that $\mc{P}' ((\phi - \epsilon , \phi +
\epsilon))$ is of finite-length.

(iii) Let $\sigma'\in\Stab(\cat{T'})$ and suppose that
$\sigma:=F^{-1}\sigma'\in\Stab(\cat{T})$. Then
$F^{-1}(\sigma'\cdot(G,f))=\sigma\cdot(G,f)\in\Stab(\cat{T})$, for any
$(G,f)\in\glpiu$.\end{remark}

\begin{lem}\label{prop:2.6}
Assume $F$ satisfies \eqref{eq:ind}. Then the subset
$${\rm Dom} (F^{-1}) := \{ \sigma' \in \Stab (\cat{T}') \, : \, \sigma = F^{-1} \sigma' \in \Stab (\cat{T}) \}$$
is closed.
\end{lem}

\begin{proof} Assume  $(\sigma')_s \to \overline{\sigma'} = (\overline{Z'}, \overline{\mc{P}'})$ in $\Stab (\cat{T}')$, with $(\sigma')_s \in {\rm Dom} (F^{-1})$. We want to prove that $\overline{\sigma} = F^{-1} \overline{\sigma'} = (\overline{Z}, \overline{\mc{P}})$ has HN-filtrations. If $0\neq E \in \cat{T}$, then in $\overline{\sigma'}$ there exists a HN-filtration of $F(E)$. Denote the semistable factors of $F(E)$ by $A_1,\ldots,A_n$, with $\overline{\phi'} (A_1) > \ldots > \overline{\phi'} (A_n)$, where $\overline{\phi'}$ denotes the phase in  $\overline{\sigma'}$.

Since $(\sigma')_s \to \overline{\sigma'}$, for $s\gg 0$ we can
assume $(\phi')_s^{\pm} (A_1) > \ldots > (\phi')_s^{\pm} (A_n)$.
By replacing each $A_i$ by its HN-filtration in $(\sigma')_s$ we
get the HN-filtration of $F(E)$ in $(\sigma')_s$. Since
$(\sigma')_s \in{\rm Dom} (F^{-1})$, this HN-filtration is, up to
isomorphism, the image via $F$ of the HN-filtration of $E$ in
$\sigma_s = F^{-1} (\sigma')_s$. Note that we are using here the
uniqueness of the semistable objects and of the morphisms in
HN-filtrations, up to isomorphism. This means that $A_i \in
F(\cat{T})$, for all $i\in\{1,\ldots,n\}$. By the uniqueness of
HN-filtrations, this proves that the HN-filtration just considered
is the image of a HN-filtration in $\sigma$ via the functor
$F$.\end{proof}

\begin{lem}\label{prop:2.5}
Assume $F$ satisfies \eqref{eq:ind}. Then the map $F^{-1}:{\rm
Dom}(F^{-1})\to\Stab(\cat{T})$ is continuous.
\end{lem}

\begin{proof}
First of all notice that, given a nonzero object $E\in\cat{T}$ and
$\sigma=F^{-1}\sigma'\in\Stab(\cat{T})$, then the image via $F$ of
the HN-filtration of $E$ with respect to $\sigma$ is the HN-filtration of $F(E)$ with respect to $\sigma'$. Hence
$\phi^+_{\sigma}(E)=\phi^+_{\sigma'}(F(E))$,
$\phi^-_{\sigma}(E)=\phi^-_{\sigma'}(F(E))$, and
$m_\sigma(E)=m_{\sigma'}(F(E))$.

As a consequence, for
$\sigma_1=F^{-1}\sigma_1',\sigma_2=F^{-1}\sigma_2'\in\Stab(\cat{T})$,
the following inequality holds:
\[
\begin{split}
d(\sigma_1,\sigma_2)&=\underset{0 \neq E \in \cat{T}}{\sup} \left\lbrace |\phi_{\sigma_2}^+ (E) - \phi_{\sigma_1}^+ (E)|, |\phi_{\sigma_2}^- (E) - \phi_{\sigma_1}^- (E)|, \left\vert\log \frac{m_{\sigma_2} (E)}{m_{\sigma_1} (E)}\right\vert \right\rbrace\\
&=\underset{0 \neq E \in \cat{T}}{\sup} \left\lbrace |\phi_{\sigma_2'}^+ (F(E)) - \phi_{\sigma_1'}^+ (F(E))|, |\phi_{\sigma_2'}^- (F(E)) - \phi_{\sigma_1'}^- (F(E))|, \left\vert\log \frac{m_{\sigma_2'} (F(E))}{m_{\sigma_1'} (F(E))}\right\vert \right\rbrace\\
&\leq\underset{0 \neq G \in \cat{T'}}{\sup} \left\lbrace |\phi_{\sigma_2'}^+ (G) - \phi_{\sigma_1'}^+ (G)|, |\phi_{\sigma_2'}^- (G) - \phi_{\sigma_1'}^- (G)|, \left\vert\log \frac{m_{\sigma_2'} (G)}{m_{\sigma_1'} (G)}\right\vert \right\rbrace\\
&=d(\sigma_1',\sigma_2').
\end{split}
\]
Thus, since the topology on the stability manifold is induced by
the generalized metric $d$, $F^{-1}$ is continuous.
\end{proof}

We now show how to construct stability conditions using special
exact functors and abelian categories. Similar existence results
will be considered in Section \ref{subsec:gen}.

\begin{definition}\label{def:Fadmis} Let $F:\cat{T}\to\cat{T}'$ be an exact functor. An abelian category $\cat{A}\subseteq\cat{T}$ is called \emph{$F$-admissible} if
\begin{itemize}
\item[(i)] $\cat{A}$ is the heart of a bounded $t$-structure;

\item[(ii)] $\Hom^{<0}_{\cat{T'}}(F(A),F(B))=0$, for all
$A,B\in\cat{A}$;

\item[(iii)] $F$ is full when restricted to $\cat{A}$.
\end{itemize}\end{definition}

For a subcategory $\cat{C}$ of $\cat{T}$, we denote by
$\langle\cat{C}\rangle\subseteq\cat{T}$ the smallest
extension-closed full subcategory containing $\cat{C}$. We now
show how to produce hearts of bounded $t$-structures.

\begin{lem}\label{lem:inducedtstr}
Let $F:\cat{T}\to\cat{T}'$ be an exact functor, and assume that
$\langle F (\cat{T}) \rangle =\cat{T'}$. Let $\cat{A} \subseteq
\cat{T}$ be an $F$-admissible abelian category. Then the
subcategory $\cat{A'}:=\langle F
(\cat{A})\rangle\subseteq\cat{T}'$ is the heart of a bounded
$t$-structure on $\cat{T}'$.
\end{lem}

\begin{proof}
Since, by assumption, the smallest triangulated subcategory of
$\cat{T'}$ containing $F(\cat{T})$ is $\cat{T'}$ itself, then, if
$0\neq E\in\cat{T'}$, there exist $M_1 , \ldots , M_k \in \cat{T}$
such that $E$ admits a filtration given by distinguished triangles
$E_{s-1}\to E_s\to F(M_s)$ ($s=1,\ldots,k$) with $E_0=0$ and
$E_k=E$. Moreover, due to the fact that $\cat{A}$ is the heart of
a $t$-structure, we may assume that $M_s=A_s[i_s]$ with
$A_1,\ldots,A_k\in\cat{A}$ and $i_1,\ldots,i_k$ integers.

We claim that there exist $B_1,\ldots,B_k\in\cat{A}$ (some of them could be zero) and
$j_1,\ldots,j_k$ integers such that $E$ can be filtered by
distinguished triangles $Q_{s-1}\to Q_s\to F(B_s[j_s])$
($s=1,\ldots,k$) with $Q_0=0$, $Q_k=E$, and $j_1\geq
j_2\geq\ldots\geq j_{k-1}\geq j_k=\min\{i_1,\ldots, i_k\}$. Notice
that $\cat{A}'$ is, by definition, a full additive subcategory of
$\cat{T}'$ and $\Hom^{<0}_{\cat{A'}} (A',B')=0$ if
$A',B'\in\cat{A'}$ . Hence by \cite[Lemma 3.2]{B1} the claim
implies  that $\cat{A}'$ is the heart of a bounded $t$-structure
on $\cat{T}$, as wanted.

\medskip

To prove the claim we proceed by induction on $k$. For $k=1$ there
is nothing to prove and so we may assume $k>1$. Consider
$E_{k-1}$. By the induction hypothesis, there exist
$B'_1,\ldots,B'_{k-1}\in\cat{A}$ and $j'_1,\ldots,j'_{k-1}$
integers such that $E_{k-1}$ can be filtered by distinguished
triangles $Q'_{s-1}\to Q'_s\to F(B'_s[j'_s])$ ($s=1,\ldots,k-1$)
with $Q'_0=0$, $Q'_{k-1}=E_{k-1}$, and $j'_1\geq\ldots\geq
j'_{k-1}=\min\{i_1,\ldots, i_{k-1}\}$. If $j'_{k-1}\geq i_k$ we
have our desired filtration by setting $B_k:=A_k$, $j_k:=i_k$,
$B_s:=B'_s$, $Q_s:=Q'_s$ and $j_s:=j'_s$ if $s=1,\ldots,k-1$.
Otherwise, we distinguish two cases.

If $i_k = j'_{k-1} +1$, let $A$ be a cone of $Q'_{k-2}\to E$. Then
we have a diagram of exact triangles
\begin{equation*}
\xymatrix{Q'_{k-2}\ar[d]\ar[r] & E\ar[d]^{\id}\ar[r] & A\ar[d]\\
          Q'_{k-1}=E_{k-1}\ar[d]\ar[r] & E\ar[d]\ar[r] & F(A_k)[i_k]\ar[d]^{g[i_k]}\\
          F(B'_{k-1})[i_k - 1]\ar[r] & 0\ar[r] & F(B'_{k-1})[i_k].\\
          }
\end{equation*}
The condition that $F$ is full when restricted to $\cat{A}$ yields
$A\cong F(D)$, where $D$ is a cone of
$A_k[i_k-1]\mor[f{[i_{k-1}]}]B'_{k-1}[i_k-1]$ and $F(f)=g$. Since
$\cat{A}$ is abelian, $D$ is the extension
$$\nk(f) [i_k] \to D \to \co(f) [i_k - 1].$$
So, the last part of the filtration becomes
\begin{equation*}
\mbox{\small
\xymatrix{Q'_{k-3}\ar[rr] && Q'_{k-2}\ar[rr]\ar[dl] && Q_{k-1} \ar[rr] \ar[dl] && E, \ar[dl] \\
          & F(B'_{k-2})[j'_{k-2}] \ar[ul]^{[1]} && F(\nk(f))[i_k] \ar[ul]^{[1]} && F(\co(f))[j'_{k-1}] \ar[ul]^{[1]}}
}
\end{equation*}
where $Q_{k-1}[1]$ is a cone of the composite map $E \to A \to
F(\co(f))[j'_{k-1}]$. Set $j_k:=j'_{k-1}$ and $B_k:=\co(f)$. Then
we have
$$\min\{i_1,\ldots,i_k\}=\min\{j'_1,\ldots,j'_{k-1},i_k\}=\min\{j'_1,\ldots,j'_{k-2},i_k,j_k\}=j_k.$$
Now consider $Q_{k-1}$. By induction there exist
$B_1,\ldots,B_{k-1}\in\cat{A}$ and $j_1,\ldots,j_{k-1}$ integers
such that $Q_{k-1}$ can be filtered by distinguished triangles
$Q_{s-1}\to Q_s\to F(B_s[j_s])$ ($s=1,\ldots,k-1$) with $Q_0=0$
and $j_1\geq\ldots\geq j_{k-1}=\min\{j'_1,\ldots,j'_{k-2},i_k\}$.
But then $j_1\geq\ldots j_{k-1}\geq j_k$ and we get our desired
filtration.

Assume instead that $i_k>j'_{k-1}+1$. Then, as before, if $A$ is a
cone of $Q'_{k-2}\to E_k$, $A$ is an extension of $F(A_k)[i_k]$ by
$F(B'_{k-1})[j'_{k-1}]$. Since by hypothesis there are no
non-trivial extensions, $A\cong F(A_{i_k})[i_k]\oplus
F(B'_{k-1})[j'_{k-1}]$ and we can filter $E$ as
\begin{equation*}
\mbox{\small
\xymatrix{
Q'_{k-3}\ar[rr] && Q'_{k-2} \ar[rr]  \ar[dl] && Q_{k-1}\ar[rr] \ar[dl] && E \ar[dl] \\
          & F(B'_{k-2})[j'_{k-2}] \ar[ul]^{[1]} && F(A_k)[i_k] \ar[ul]^{[1]} && F(B'_{k-1})[j'_{k-1}], \ar[ul]^{[1]}}
}
\end{equation*}
where $Q_{k-1}[1]$ is a cone of the composite map $E \to A \to
F(B'_{k-1})[j'_{k-1}]$. Set $j_k:=j'_{k-1}$ and $B_k:=B'_{k-1}$.
Then we have
$$\min\{i_1,\ldots,i_k\}=\min\{j'_1,\ldots,j'_{k-1},i_k\}=\min\{j'_1,\ldots,j'_{k-2},i_k,j_k\}=j_k.$$
As in the previous case, we can now conclude by induction.
\end{proof}

Let us move to the problem of inducing stability conditions using
an exact functor.

\begin{prop}\label{pro:existence}
Let $F: \cat{T} \to \cat{T}'$ be an exact functor which satisfies
\eqref{eq:ind} and assume that $\langle F (\cat{T}) \rangle =
\cat{T'}$. Let $\sigma = (Z, \mc{P}) \in \Stab(\cat{T})$ be such
that its heart $\mc{P}((0,1])$ is of finite length with a finite
number of minimal objects. Assume furthermore that $\mc{P}((0,1])$
is $F$-admissible. Then $F_*:K(\cat{T})\to K(\cat{T'})$ is an
isomorphism. Define $\sigma' = F(\sigma)=(Z' , \mc{P}')$, where
$Z' = Z \circ F_*^{-1}$ and $\mc{P}' ((0, 1]) = \langle F(\mc{P}
((0, 1]))\rangle$. Then $\sigma'$ is a locally finite stability
condition on $\cat{T}'$. Moreover, $F^{-1} \sigma' = \sigma$.
\end{prop}

\begin{proof} First of all notice that $\mc{P}' ((0, 1])$ is of finite length
and has a finite number of minimal objects, which are nothing but
the images via $F$ of the minimal objects of $\mc{P} ((0, 1])$.
Indeed, if $S\in\kp((0,1])$ is minimal and $A'\mor[l']F(S)$ is a
monomorphism, then, by Lemma \ref{lem:accessory1}, $A'\cong F(A)$
and $l'=F(l)$, for $l:A\to S$. But then, since $F$ satisfies
\eqref{eq:ind}, $l$ is a monomorphism too. Hence either $A\cong 0$
or $A\cong S$. So $F(S)$ minimal. The fact that $\mc{P}' ((0, 1])$
is generated by its minimal objects follows now from its own
definition. As a consequence, $F_* : K (\cat{T}) \to K(\cat{T}')$
is an isomorphism and the definition of $\sigma'$ has meaning.

Now the first part of the proposition follows form the previous
lemma and from Remark \ref{rmk:tstruct}, since $\mc{P}' ((0, 1])$
is an abelian category of finite length. To prove that $F^{-1}
\sigma' = \sigma$, we only have to show that if $E \in \cat{T}$ is
$\sigma$-semistable, then $F(E)$ is $\sigma'$-semistable. This
follows again from Lemma \ref{lem:accessory1}.
\end{proof}

\begin{lem}\label{lem:accessory1}
Let $F:\cat{A}\to\cat{A'}$ be a full exact functor between abelian
categories. Assume that $\langle F(\cat{A})\rangle=\cat{A'}$. Then
$F(\cat{A})$ is closed under subobjects and quotients.
\end{lem}

\begin{proof}
Let $0\to M\mor[f]F(A)$, with $M\in\cat{A'}$. Since $\cat{A'}$ is
generated by $F(\cat{A})$ by extensions, there exist
$A_1,\ldots,A_k\in\cat{A}$ such that  $M$ is an extension of
$F(A_1),\ldots,F(A_k)$. We want to show that there exists
$E\in\cat{A}$ such that $M\cong F(E)$. This is enough to conclude
the proof.

We proceed by induction on $k$. For $k=0$ there is nothing to
prove. Suppose $k>0$. Let $N$ be the kernel of the morphism
$M\twoheadrightarrow F(A_k)$. Then $N$ is a subobject of $F(A)$
which is an extension of $F(A_1),\ldots,F(A_{k-1})$. By the
inductive assumption, $N\cong F(B)$, for some $B\in\cat{A}$. Hence
we have a short exact sequence
\[
0\lto F(B)\lto M\lto F(A_k)\lto 0.
\]
We have the following diagram
\begin{equation*}
\xymatrix{
    & 0\ar[d] & 0\ar[d] & 0\ar[d] & \\
    0\ar[r] & F(B)\ar[d]^{\id}\ar[r] & M\ar[d]^{f}\ar[r] & F(A_k)\ar[d]\ar[r] & 0\\
    0\ar[r] & F(B)\ar[d]\ar[r]^{g} & F(A)\ar[d]\ar[r] & \co(g)\ar[d]\ar[r] & 0\\
         & 0\ar[r] & \co(f)\ar[r]^{\id}\ar[d] & \co(f)\ar[r]\ar[d] & 0.\\
         & & 0 & 0 &
          }
\end{equation*}
Since $F$ is full, $g=F(h)$, for $h:B\to A$ and so $\co(g)\cong
F(\co(h))$. Then, again since $F$ is full, $\co(f)\cong F(L)$ for
some $L\in\cat{A}$. As a consequence, in a similar way, $M\cong
F(E)$ for some $E\in\cat{A}$.
\end{proof}

\medskip

The result in Proposition \ref{pro:existence} will be applied in a
geometric context (Section \ref{sec:CY}) where a weaker form of it
would be enough. Notice that the latter is also a consequence of
more general results in \cite{Pol} and, for the convenience of the
reader, we include the precise statement here (although it will
not be explicitly used in the rest of this paper).

Let $\cat{\widetilde{T}}$ and $\cat{\widetilde{T}'}$ be two
triangulated categories in which all small coproducts exist and
let $\cat{T}\subseteq\cat{\widetilde{T}}$ and
$\cat{T'}\subseteq\cat{\widetilde{T}'}$  be full triangulated
essentially small subcategories. Consider an exact functor
$F:\cat{\widetilde{T}}\to\cat{\widetilde{T}'}$ such that:
\begin{itemize}
\item $F$ commutes with small coproducts; \item $F$ has a left
adjoint $G:\cat{\widetilde{T}'}\to\cat{\widetilde{T}}$; \item
$F(\cat{T})\subseteq\cat{T'}$ and if $E\in\cat{\widetilde{T}}$ and
$F(E)\in\cat{T'}$, then $E\in\cat{T}$; \item
$G(\cat{T'})\subseteq\cat{T}$; \item the induced functor
$F:\cat{T}\to\cat{T'}$ satisfies condition \eqref{eq:ind}.
\end{itemize}
We can now state the following:

\begin{thm}\label{thm:Poli} {\bf (\cite{Pol})} Let $\sigma'=(Z',\kp')\in\Stab(\cat{T'})$ be such that $FG(\kp'(\phi))\subseteq\kp'((\phi,+\infty))$, for all $\phi\in\RR$. Then $F^{-1}\sigma'\in\Stab(\cat{T})$.\end{thm}

\begin{proof} Due to what we have already observed, it is enough to show that, in $\sigma:=F^{-1}\sigma'$, every $E\in\cat{T}$ admits a HN-filtration.

Let $A'\to F(E)\to B'$ be the last triangle in the HN-filtration
of $F(E)$ in $\sigma'$, with $A'\in\kp'((\phi,+\infty))$ and
$B'\in\kp'(\phi)$. Let $I\subseteq\RR$ be an interval. Define
\[
\kp(I):=\{C\in\cat{T}:F(C)\in\kp'(I)\}.
\]
Under our assumptions \cite[Thm.\ 2.1.2]{Pol} applies. Thus the
pair $(\kp((\phi,+\infty)),\kp((-\infty,\phi]))$ defines a
$t$-structure on $\cat{T}$. Hence there exists a triangle $A\to
E\to B$ in $\cat{T}$, where $A\in\kp((\phi,+\infty))$ and
$B\in\kp((-\infty,\phi])$. Applying the functor $F$, we get
$F(A)\to F(E)\to F(B)$ in $\cat{T'}$, with
$F(A)\in\kp'((\phi,+\infty))$ and $F(B)\in\kp'((-\infty,\phi])$.

By uniqueness of the HN-filtration, $A'\cong F(A)$ and $B'\cong
F(B)$. Hence $B\in\kp(\phi)$ and, proceeding further with $A$, we
get the existence of the HN-filtration for $E$ in
$\sigma$.\end{proof}

\bigskip

\subsection{Invariant stability conditions}\label{subsec:gen}~

\medskip

Let $X$ be a smooth projective variety over $\CC$ with an action
of a finite group $G$. We denote by $\coh_G(X)$ the abelian
category of $G$-equivariant coherent sheaves on $X$, i.e.\ the
category whose objects are pairs $(\ke,\{\lambda_g\}_{g\in G})$,
where $\ke\in\coh(X)$ and, for any $g_1,g_2\in G$,
$\lambda_{g_i}:\ke\isomor g_i^*\ke$ is an isomorphism such that
$\lambda_{g_1g_2}=g_2^*(\lambda_{g_1})\circ\lambda_{g_2}$. The set
of these isomorphisms is a \emph{$G$-linearization} of $\ke$ (very
often a $G$-linearization will be simply denoted by $\lambda$).
The morphisms in $\coh_G(X)$ are just the morphisms of coherent
sheaves compatible with the $G$-linearizations (for more details
see, for example, \cite{BL,BKR}). We put
$\Db_G(X):=\Db(\coh_G(X))$. Since $G$ is finite, when needed,
$\Db_G(X)$ can equivalently be described in terms of
$G$-equivariant objects in $\Db(X)$ (see, for example,
\cite[Sect.\ 1.1]{Pl}).

The main aim of this section is to prove Theorem \ref{thm:main2}.
To this end, consider the functors
\[
\Forg_G:\Db_G(X)\lto\Db(X)
\]
which forgets the $G$-linearization, and
\[
\Inf_G:\Db(X)\lto\Db_G(X)
\]
defined by $$\Inf_G(\ke):=\left(\bigoplus_{g\in
G}g^*\ke,\lambda_\mathrm{nat}\right),$$ where
$\lambda_\mathrm{nat}$ is the natural $G$-linearization. These
functors are adjoint (see \cite[Sect.\ 8]{BL}):
\[
\Hom_{\Db_G(X)}(\Inf_G(\ke),(\kf,\beta))=\Hom_{\Db(X)}(\ke,\Forg_G((\kf,\beta))).
\]
Since both functors are faithful, the results of Section
\ref{subsec:stabex} apply and they induce stability conditions on
$\Db_G(X)$ and $\Db(X)$.

The group $G$ acts on $\Stab(\Db(X))$ and $\NStab(\Db(X))$ in the
obvious manner via the autoequivalences $\{g^*:g:\in G\}$. Hence
we can define the subset
\[
\Gamma_X:=\{\sigma\in\Stab(\Db(X)):g^*\sigma=\sigma\mbox{, for any
}g\in G\}.
\]

\begin{lem}\label{lem:clossub} $\Gamma_X$ is a closed submanifold of $\Stab(\Db(X))$ such that the diagram
\[
\xymatrix{\Gamma_X\cap\Sigma\ar@{^{(}->}[rr]\ar[d]_{\kz|_{\Gamma_X}}&
& \Sigma\ar[d]^{\kz}\\ (V(\Sigma))_G\ar@{^{(}->}[rr]& & V(\Sigma)}
\]
commutes and $\kz|_{\Gamma_X}$ is a local homeomorphism. Here
$(-)_G$ is the $G$-invariant part and $\Sigma$ is a connected
component of $\Stab(\Db(X))$. \end{lem}

\begin{proof} The subset $\Gamma_X$ is closed because $\Gamma_X=\bigcap_{g\in G}(g,\id)^{-1}(\Delta)=(\Stab(\Db(X)))_G$, where $\Delta$ is the diagonal in $\Stab(\Db(X))\times\Stab(\Db(X))$, and $g$ acts continuously on $\Stab(\Db(X))$. Consider now the map $\kz|_{\Gamma_X}:\Gamma_X\to (V(\Sigma))_G$ and let $\sigma\in\Gamma_X$. It is enough to prove that $\kz|_{\Gamma_X}$ is a local homeomorphism in a neighbourhood of $\sigma$.

Take the open subset $U:=\{\sigma'\in\Stab(\Db(X)):
d(\sigma',\sigma)<1/2\}$ where $\kz|_U$ maps onto $V\subseteq
V(\Sigma)$. Then $\Gamma_X\cap U$ goes homeomorphically into
$V\cap (V(\Sigma))_G$. Indeed, $\kz(U\cap\Gamma_X)\subseteq V\cap
(V(\Sigma))_G$ and, if $Z'\in V\cap (V(\Sigma))_G$, take
$\sigma'\in U$ such that $\kz(\sigma')=Z'$. We now have
\[
d(\sigma',g^*\sigma')\leq
d(\sigma',\sigma)+d(g^*\sigma',\sigma)=d(\sigma',\sigma)+d(\sigma',\sigma)<1,
\]
where the last equality holds since $\sigma$ is in $\Gamma_X$ and
$G$ acts by isometries on the stability manifold. Moreover
$\kz(g^*\sigma')=Z'\circ (g^*)^{-1}=Z'=\kz(\sigma')$. Applying
\cite[Lemma 6.4]{B1}, we get $\sigma'=g^*\sigma'$ so that
$\sigma'\in U\cap\Gamma_X$.
\end{proof}

The following lemma is proved in \cite{Pol} in a more general form
and it follows from Theorem \ref{thm:Poli}. For the convenience of
the reader we outline the proof in our special case. We keep the
notation of Section \ref{subsec:stabex}.

\begin{lem}\label{lem:welldef} If $\sigma=(Z,\kp)\in\Gamma_X$, then $\Forg_G^{-1}\sigma\in\Stab(\Db_G(X))$.\end{lem}

\begin{proof} As we pointed out in Remark \ref{rmk:HN}, it is sufficient to show that in $\Forg_G^{-1}\sigma$ any object has a HN-filtration. So, let $(\ke,\lambda)\in\Db_G(X)$ and take
\[
\xymatrix{\ke'\ar[r]^{s}&\ke\ar[r]^{t}&\ka}
\]
to be the last triangle in the HN-filtration of
$\ke(=\Forg_G((\ke,\lambda)))\in\Db(X)$ in $\sigma$, where
$\ka\in\kp(\phi)$ and $\ke'\in\kp(>\phi)$.

Since, by assumption, for any $g\in G$, we have
$g^*\ke'\in\kp(>\phi)$ and $g^*\ka\in\kp(\phi)$, it follows that
$\Hom_{\Db(X)}(\ke',g^*\ka)=0$. Hence, for any $g\in G$,  there
exist unique morphisms $\beta^1_g:\ke'\isomor g^*\ke'$ and
$\beta^2_g:\ka\isomor g^*\ka$, making the diagram
\[
\xymatrix{\ke'\ar[d]_{\beta^1_g}\ar[rr]^{s}&&\ke\ar[d]_{\lambda_g}\ar[rr]^{t}&&\ka\ar[d]_{\beta^2_g}\\
g^*\ke'\ar[rr]^{g^*s}&&g^*\ke\ar[rr]^{g^*t}&&g^*\ka}
\]
commutative. Notice that, by uniqueness,
$\beta^i_{hg}=h^*(\beta^i_g)\circ\beta^i_h$ and $\beta_\id^i=\id$,
for any $g,h\in G$ and $i\in\{1,2\}$. Moreover, the same argument
shows that $\beta^i_g$ is an isomorphism.

Since $(\ka,\beta^2)\in\Forg_G^{-1}\kp(\phi)$, proceeding further
with the object $(\ke',\beta^1)$ we get a HN-filtration for
$(\ke,\lambda)$.\end{proof}

This allows us to introduce a map
$\Forg_G^{-1}:\Gamma_X\to\Stab(\Db_G(X))$.

\begin{prop}\label{prop:clossubman} The morphism $\Forg^{-1}_G$ just defined is continuous and the subset $\Forg^{-1}_G(\Gamma_X)$ is a closed embedded submanifold.\end{prop}

\begin{proof} The first part of the statement follows from Lemma \ref{lem:welldef} and Lemma \ref{prop:2.5}.

To prove the second assertion, consider the functor $\Inf_G$. If
$\sigma=(Z,\kp)\in\Gamma_X$ and $\sigma'=\Forg_G^{-1}(\sigma)$,
then
$\Inf_G^{-1}\sigma'=(|G|\cdot\id_{\glpiu})\sigma\in\Stab(\Db(X))$.
This can be easily shown by recalling that, by definition, for any
$e\in K(X)\otimes\CC$, the following series of equalities holds
true
\[
(\Inf_G^{-1}\circ\Forg_G^{-1})Z(e)=Z(\Forg_{G*}(\Inf_{G*}(e)))=Z\left(\sum_{g\in
G}g^*e\right)=|G|Z(e).
\]
Moreover, on the level of slicings, we get the following
description
\[
(\Inf_G^{-1}\circ\Forg_G^{-1})(\kp)(\phi)=\left\lbrace\ke\in\Db(X):\bigoplus_{g\in
G}g^*\ke\in\kp(\phi)\right\rbrace,
\]
for any $\phi\in\RR$. Since $\kp(\phi)$ is closed under direct
summands,
$(\Inf_G^{-1}\circ\Forg_G^{-1})(\kp)(\phi)\subseteq\kp(\phi)$. On
the other hand, the fact that $g^*\sigma=\sigma$ implies that if
$\ke\in\kp(\phi)$, then $\bigoplus_{g\in G}g^*\ke\in\kp(\phi)$.
Hence
$\Inf_G^{-1}\sigma'=(|G|\cdot\id_{\glpiu})\sigma\in\Stab(\Db(X))$.
See also \cite[Prop.\ 2.2.3]{Pol}.

Using this fact and Lemma \ref{prop:2.5}, we get that
$\Inf_G^{-1}$ is continuous in its domain and that the morphism
$\Inf^{-1}_G\circ\Forg_G^{-1}:\Gamma_X\to\Gamma_X$ is an
isomorphism. In particular $\Forg^{-1}_G(\Gamma_X)$ is an embedded
submanifold.

To prove that $\Forg^{-1}_G(\Gamma_X)$ is also closed, take a
sequence $\sigma'_n:=\Forg_G^{-1}\sigma_n\in\Stab(\Db_G(X))$
converging to $\overline\sigma'$. We have to show that
$\overline\sigma'=\Forg_G^{-1}\overline\sigma$, for some
$\overline\sigma\in\Gamma_X$.

Since $\Inf^{-1}_G\overline\sigma'\in\Stab(\Db(X))$, applying
$\Inf^{-1}_G$ to the previous sequence, we have that the sequence
$\Inf_G^{-1}\sigma'_n$ converges to $\Inf_G^{-1}\overline\sigma'$
in $\Stab(\Db(X))$. By what we have proved before, we have
$\Inf_G^{-1}\sigma'_n=(|G|\cdot\id_{\glpiu})\sigma_n\in\Gamma_X$
and so $\Inf^{-1}_G\overline\sigma'=\overline\tau\in\Gamma_X$,
because $\Gamma_X$ is closed.

In particular, $\sigma_n$ converges to
$(1/|G|\cdot\id_{\glpiu})\overline\tau$ in $\Gamma_X$. Applying
now $\Forg_G^{-1}$, we get that the sequence
$\sigma'_n=\Forg_G^{-1}\sigma_n$ converges to
$\Forg_G^{-1}\overline\sigma$. By uniqueness of limits,
$\overline\sigma'=\Forg^{-1}_G\overline\sigma$.\end{proof}

The previous results show that the subset
$\Gamma_X=(\Stab(\Db(X)))_G$ of $G$-invariant stability conditions
on $\Db(X)$ embeds as a closed submanifold into $\Stab(\Db_G(X))$
by the forgetful functor, as stated in Theorem \ref{thm:main2}.

\begin{remark}\label{rmk:numstab} All the previous arguments remain
true when dealing with numerical stability conditions, just
substituting $(K(X)\otimes\CC)\dual_G$ and $(K(X)\otimes\CC)\dual$
with $(\kn(X)\otimes\CC)_G$ and $\kn(X)\otimes\CC$. More
precisely, the Euler--Poincar\'e pairing $\chi$ is non-degenerate
and $G$-invariant. Hence it gives a canonical identification
$(\kn(X)\otimes\CC)\dual_G\cong(\kn(X)\otimes\CC)_G$. (Here
$(-)_G$ stands again for the $G$-invariant part.) Moreover, the
map $\Forg_{G*}$ factorizes in the following way
\[
\xymatrix@1{\kn(\Db_G(X))\ar@/_2pc/[rr]_{\Forg_{G*}}\ar@{>>}[r]
&(\kn(X)\otimes\CC)_G\;\ar@{^{(}->}[r]&\kn(X)\otimes\CC.}
\]
As a consequence, we also get a submanifold $\Gamma_X$ in
$\NStab(\Db(X))$ with the properties stated in Theorem
\ref{thm:main2}.\end{remark}

\bigskip

\subsection{Examples}\label{subsec:exinvar}~

\medskip

We conclude this section with a discussion of the first geometric
cases where the previous results apply.

\subsubsection{Weighted projective lines}\label{ex:projline}
Suppose $E$ is an elliptic curve carrying an action of a finite
group $G$. Since the even cohomology of $E$ must remain fixed
under this action, the standard stability condition given by slope
stability on coherent sheaves is invariant under $G$. Thus, as
$\NStab(\Db(E))$ is a $\glpiu$-orbit by \cite[Thm.\  9.1]{B1}, we
have $(\NStab(\Db(E)))_G=\NStab(\Db(E))$ and Theorem
\ref{thm:main2} implies that $\NStab(\Db(E))$ is embedded as a
closed submanifold into $\NStab(\Db_G(E))$.

Section 5.8 in \cite{GL} presents examples consisting of an elliptic curve $E$ and an involution
$\iota$ such that the category of coherent sheaves of the stack $[E/\langle\iota\rangle]$ and that of a certain weighted projective line $C$ are equivalent. These examples are studied more
recently by Ueda in \cite{U}, where he proves (at least for two special classes of these examples)
that the derived category  $\Db(\coh(C))$
is equivalent to the derived Fukaya category of the elliptic singularity associated to $E$.
As mirror symmetry predicts an isomorphism between the complex and K\"{a}hler
moduli of mirror pairs, one thus expects $\NStab(\Db_{\langle\iota\rangle}(E))=
\NStab(\Db(\coh(C)))$ to be related  to
the unfolding space of the elliptic singularity.
The embedding of stability manifolds of the previous paragraph makes this
picture seem further plausible, for $\NStab(\Db(E))$ itself has the following
geometric interpretation: $\NStab(\Db(E))/\Aut(\Db(E))$ is a $\CC^*$-bundle over the
modular curve (\cite{B1}). Of course, one could try to understand
the geometry of $\NStab(\Db_{\langle\iota\rangle}(E))$ directly by making use of the
well-known Beilinson-type result giving a complete strong exceptional collection in $\Db(\coh(C))$.

Related examples have been studied in \cite{KST}.

\subsubsection{Kummer surfaces}\label{ex:Kummer} Let $A$ be
an abelian surface and $\Km(A)$ the associated Kummer surface
(i.e.\ the minimal resolution of the quotient $A/\langle
\iota\rangle$, where $\iota:A\isomor A$ is the involution such
that $\iota(a)=-a$). In \cite{HMS} it was proved that
$\NDStab(\Db(A))$ (Example \ref{ex:K3ab}) is the unique connected
component in $\NStab(\Db(A))$ of maximal dimension equal to
$\rk\NS(A)+2$. Since $\iota^*$ acts as the identity on harmonic $2$-forms,
$\iota^*:\kn(A)\isomor\kn(A)$ is the identity. Hence $\Gamma_A$ is
open and closed in $\NStab(\Db(A))$, since
$(\kn(A))_{\langle\iota\rangle}=\kn(A)$. Consequently, if
non-empty, $\Gamma_A$ is a connected component.

Take $\sigma_{\omega,\beta}\in\NDStab(\Db(A))$, determined by the
linear function $Z_{\omega,\beta}$ and the abelian category
$\ka(\omega,\beta)$, according to Bridgeland's definition in
Example \ref{ex:K3ab}. By the above remarks $\beta$ and $\omega$
are invariant and so $Z_{\omega,\beta}$ and $\ka(\omega,\beta)$
are invariant as well. In particular,
$\iota^*\sigma_{\omega,\beta}=\sigma_{\omega,\beta}$. Thus
$\Gamma_A=\NDStab(\Db(A))$ and, by Theorem \ref{thm:main2},
$\NDStab(\Db(A))$ is realized as a closed submanifold of
$\NDStab(\Db(\Km(A)))$, the connected component of
$\NStab(\Db(\Km(A))$ defined in Example \ref{ex:K3ab}.

\subsubsection{K3 surfaces with cyclic
automorphisms}\label{ex:K3autom} Let $X$ be a K3 surface with an
automorphism $f:X\to X$ of finite order $n$. If $G\cong\ZZ/n\ZZ$
is the group generated by $f$, then the quotient $Y:=X/G$ has a
finite number of singular points. If $f^*\sigma_X=\sigma_X$, where
$H^{2,0}(X)=\langle\sigma_X\rangle$, then $f$ is a
\emph{symplectic automorphism} and the minimal crepant
desingularization $Z$ of $Y$ is a K3 surface. In this case
$\rk\kn(X)_G<\rk\kn(Z)$. Indeed, by the discussion in \cite[Sect.\
3]{Mo}, $\Pic(X)_G$ is embedded in $\Pic(Z)\otimes\CC$ and its
orthogonal complement contains the (non-trivial) sublattice
generated by the curves obtained desingularizing $Y$. So
$\Gamma_X$ is a closed submanifold of $\NStab^\dag(\Db(Z))$ of
strictly smaller dimension.

On the other hand, there are cases where $f$ is not symplectic and
$f^*|_{\NS(X)}=\id_{\NS(X)}$. Examples are given by generic
elliptic fibrations $Y$ with a section. The natural involution
$\iota:Y\to Y$ obtained by sending a point $p$ on a fiber to $-p$
on the same fiber yields an automorphism with the desired
properties (indeed $\NS(X)$ is generated by the classes of the generic
fiber and of the section which are fixed by the involution).

Notice that, due to the argument in the proof of \cite[Thm.\
3.1]{Ni1}, all the desingularizations $Z$ of the quotient $X/G$,
with $f$ non-symplectic, are such that $H^{2,0}(Z)=0$, marking a
deep difference with the symplectic case. Moreover, if $f$ is not
symplectic, reasoning as in Example \ref{ex:Kummer}, one shows
that there exists a closed embedding
\[
\NStab^\dag(\Db(X))\hookrightarrow\NStab(\Db_G(X))\hookrightarrow\Stab(\Db_G(X))\cong\Stab(\Db([X/G])).
\]
One would expect $\Db([X/G])$ to behave like the derived category
of a weighted projective space. Those categories are quite well
understood and stability conditions can be possibly constructed
using the techniques in Example \ref{ex:PN}. Hence, hopefully,
this might give a different understanding of the stability
conditions on $\Db(X)$.

\section{Enriques surfaces}\label{sec:equiv}

\bigskip

In this section we prove Theorem \ref{thm:main1}. In particular,
we show that a connected component of the stability manifold of an
Enriques surface embeds as a closed submanifold into the stability
manifold of the associated K3 surface. In Section
\ref{subsec:autoeq} we improve the derived version of the Torelli
Theorem for Enriques surfaces (see \cite{BM}) and relate the
topology of the connected component just introduced to the
description of the group of autoequivalences. Finally, in Section
\ref{subsec:GeomEnr}, we treat the case of generic Enriques
surfaces $Y$ where the distinguished connected component of the
stability manifold of the Enriques surface is isomorphic to the
connected component of the space of stability conditions on the
universal cover described by Bridgeland. We also show that the
derived category of $Y$ does not contain any spherical objects,
concluding with a description of its (strongly) rigid objects.

\bigskip

\subsection{A distinguished connected component}\label{subsec:enr}~

\medskip

Let $Y$ be an Enriques surface, $\pi:X\to Y$ its universal cover
and $\iota:X\to X$ the fixed-point-free involution such that
$Y=X/G$, where $G$ is now the group generated by $\iota$. In this
special setting, $\coh(Y)$ is naturally isomorphic to the abelian
category $\coh_G(X)$. Notice that, via this equivalence, the
canonical bundle $\omega_Y$ is identified with the $G$-equivariant
sheaf $(\ko_X,-\id)$. The equivalence $\coh(Y)\cong\coh_G(X)$,
obviously yields an equivalence $\Db(Y)\cong\Db_G(X)$, which will
be used without mention for the rest of this paper. Notice that
with this identification, $\Forg_G=\pi^*$.


Consider the connected component $\NDStab(\Db(X))\subseteq\NStab(\Db(X))$ described in \cite[Thm.\ 1.1]{B2} and Example \ref{ex:K3ab}. As a consequence of Theorem \ref{thm:main2}, we can now prove the following:

\begin{prop}\label{prop:connEnr} The non-empty subset $\Sigma(Y):=\Forg_G^{-1}(\Gamma_X\cap\NDStab(\Db(X)))$ is open and closed in $\NStab(\Db(Y))$ and it is embedded as a closed submanifold into $\NDStab(\Db(X))\subseteq\NStab(\Db(X))$ via the functor $\Inf_G$. Moreover, the diagram
\begin{eqnarray}\label{eqn:commdiagr}
\xymatrix@1{\Gamma_X\cap\NDStab(\Db(X))\ar[d]\ar[r]^(.65){\Forg^{-1}_G}&\Sigma(Y)\ar[d]_{\kz}\ar[r]^(.35){\Inf_G^{-1}}&\Gamma_X\cap\NDStab(\Db(X))\ar[d]\\
(\kn(X)\otimes\CC)\dual_G\ar[r]^{\Forg\dual_{G*}}&(\kn(Y)\otimes\CC)\dual\ar[r]^{\Inf\dual_{G*}}&(\kn(X)\otimes\CC)\dual_G}
\end{eqnarray}
commutes.\end{prop}

\begin{proof} Consider the map $\Forg_G^{-1}:\Gamma_X\to\NStab(\Db(Y))$. Notice that $\Gamma_X$ is non-empty. Indeed, by \cite{Ho1, Ho2}, we can choose $\beta$ and $\omega$ as in Example \ref{ex:K3ab} which are invariant for the action of $\iota^*$. So $\iota^*\sigma_{\omega,\beta}=\sigma_{\omega,\beta}$ as in Example \ref{ex:Kummer}.

By Theorem \ref{thm:main2}, $\Forg_G^{-1}(\Gamma_X\cap\NDStab(\Db(X)))$ is closed. On the other hand, notice that $\Inf_{G*}:(\kn(X)\otimes\CC)_G\to\kn(Y)\otimes\CC$ is an isomorphism. Indeed $\Inf_{G*}=\pi_*$ and then the result follows, for example, from \cite[Lemma 3.1]{Mo}. Hence $\Forg_{G*}$ is an isomorphism as well by Remark \ref{rmk:numstab} and $$\dim(\Gamma_X\cap\NDStab(\Db(X)))=\dim(\kn(Y)\otimes\CC)=\dim(\Forg^{-1}_G(\Gamma_X\cap\NDStab(\Db(X)))).$$
Thus $\Forg^{-1}_G(\Gamma_X\cap\NDStab(\Db(X)))$ has the desired property and the diagram in the statement commutes by definition.\end{proof}

We denote by $\NStab^\dag(\Db(Y))$ the (non-empty) connected component of $\Sigma(Y)$ containing the images via $\Forg_G^{-1}$ of the stability conditions $(Z_{\omega,\beta},\ka(\omega,\beta))$ defined in Example \ref{ex:K3ab}, with $G$-invariant $\omega,\beta\in\NS(X)\otimes\QQ$. (By \cite[Prop.\ 11.2]{B2}, it is not difficult to see that the $G$-invariant stability conditions in Example \ref{ex:K3ab} are contained in a $G$-invariant connected subset of $\NStab^\dag(\Db(X))$.)

\begin{remark}\label{rmk:Bri} (i) It is perhaps worth noticing that $\NStab^\dag(\Db(Y))$ could be alternatively obtained by repeating the same construction as in \cite{B2} for Enriques surfaces. Since we will not need this in the sequel, the easy check is left to the reader.

(ii) It is not difficult to see that the functor $(-)\otimes\omega_Y$ preserves $\Sigma(Y)$.\end{remark}

\begin{ex}\label{ex:Keum} Take two non-isogenous elliptic curves $E_1$ and $E_2$ and choose two order-$2$ points $e_1\in E_1$ and $e_2\in E_2$. As remarked in \cite[Ex.\ 3.1]{K}, the abelian surface $A:=E_1\times E_2$ has an involution $\iota$ defined by
\[
\iota:(z_1,z_2)\longmapsto(-z_1+e_1,z_2+e_2).
\]
Notice that, since $\NS(A)$ is generated by the elliptic curves $E_1$ and $E_2$, $\iota^*|_{\NS(A)}=\id_{\NS(A)}$ while $\iota$ acts freely on the subgroup of order-$2$ points of $A$. Hence the induced involution $\tilde\iota:\Km(A)\to\Km(A)$ has no fixed points and the $\CC$-linear extension of $\tilde\iota$ restricted to the vector space $\NS(A)\otimes\CC\subset\NS(\Km(A))\otimes\CC$ is the identity.

If $Y$ is the Enriques surface $\Km(A)/\langle\tilde\iota\rangle$, combining Proposition \ref{prop:connEnr} and Example \ref{ex:Kummer} we obtain a connected component $\NStab^\dag(\Db(Y))\subseteq\NStab(\Db(Y))$ and embeddings
\[
\NStab^\dag(\Db(A))\hookrightarrow\NStab^\dag(\Db(Y))\hookrightarrow\NStab^\dag(\Db(\Km(A)))
\]
of closed submanifolds. More generally, a result of Keum \cite[Thm.\ 2]{K} shows that any Kummer surface is the universal cover of an Enriques surface.\end{ex}

\bigskip

\subsection{The group of autoequivalences}\label{subsec:autoeq}~

\medskip

As pointed out by Bridgeland in \cite{B2} for K3 surfaces, the
knowledge of some topological features of a special connected
component of the manifold parametrizing stability conditions can
give important information about the group of autoequivalences of
the derived category.

In this section we want to carry out the same strategy for
Enriques surfaces. The first step consists in proving a derived
version of the Torelli Theorem for Enriques surfaces first stated
in \cite{Ho1,Ho2} (see also \cite{BPV}). Notice that for K3
surfaces a Derived Torelli Theorem for K3 surfaces is already
present in the literature (see \cite{Or1} for the main result and
\cite{HLOY,Hu1,HS1,P} for further refinements). Similarly, for
Enriques surfaces, the following derived version of the classical
Torelli Theorem is available:

\begin{prop}\label{prop:BM}{\bf (\cite{BM})}
Let $Y_1$ and $Y_2$ be Enriques surfaces and let $X_1$ and $X_2$
be the universal covers endowed with the involutions $\iota_1$ and
$\iota_2$.

{\rm (i)} Any equivalence $\Db(Y_1)\cong\Db(Y_2)$ induces an
orientation preserving equivariant Hodge isometry $\widetilde
H(X_1,\ZZ)\cong\widetilde H(X_2,\ZZ)$.

{\rm (ii)} Any orientation preserving equivariant Hodge isometry
$\widetilde H(X_1,\ZZ)\cong\widetilde H(X_2,\ZZ)$ lifts to an
equivalence $\Db(Y_1)\cong\Db(Y_2)$.\end{prop}

Recall that, for a K3 surface $X$, the Mukai lattice $\widetilde
H(X,\ZZ)$ of X is the total cohomology group $H^*(X,\ZZ)$ endowed
with the Hodge and lattice structure defined, for example, in
\cite[Sect.\ 10.1]{Hu1}. The lattice $\widetilde H(X,\ZZ)$ has
signature $(4,20)$ and an isometry of $\widetilde H(X,\ZZ)$ is
\emph{orientation preserving} if it preserves the orientation of
the four positive directions in $\widetilde H(X,\RR)$. We will
denote by $\OO(\widetilde H(X,\ZZ))$ (respectively
$\OO_+(\widetilde H(X,\ZZ))$) the group of Hodge isometries
(respectively orientation preserving Hodge isometries) of
$\widetilde H(X,\ZZ)$. The subgroups consisting of the equivariant
isometries are denoted by $\OO(\widetilde H(X,\ZZ))_G$ and
$\OO_+(\widetilde H(X,\ZZ))_G$. Obviously, $G$ is naturally a
normal subgroup of $\OO(\widetilde H(X,\ZZ))_G$.

The aim of this section is to improve Proposition \ref{prop:BM}
giving a description of the Hodge isometries induced by all
possible Fourier--Mukai equivalences. As a first step, the
following result, partially relying on \cite{HMS1}, gives a first
description of the group of autoequivalences of an Enriques
surface.

\begin{prop}\label{thm:DTT} Let $Y$ be an Enriques surface and let $X$ and $G$ be as above. There
exists a natural morphism of groups
\[
\Pi:\Aut(\Db(Y))\lto\OO(\widetilde H(X,\ZZ))_G/G
\]
whose image is the index-2 subgroup $\OO_+(\widetilde H(X,\ZZ))_G/G$.
\end{prop}

\begin{proof} The easy part is defining the morphism $\Pi$ in the statement. Consider the following set of objects
\[
\KE(\Db(X)):=\{(\kg,\lambda)\in\Db_{G_\Delta}(X\times X):\Phi_\kg\in\Aut(\Db(X))\}
\]
and take $\Aut(\Db(X))_G:=\{\Phi\in\Aut(\Db(X)):\iota^*\circ\Phi\circ\iota^*\cong\Phi\}$. Here $G_\Delta$ is the group generated by the involution $\iota\times\iota$.

As pointed out in \cite[Sect.\ 3.3]{Pl}, the functors
$\Forg_{G_\Delta}$ and $\Inf_{G_\Delta}$ are $2:1$, so that we
have the following diagram
\[
\xymatrix{ &\KE(\Db(X))\ar[dl]_{\Forg_{G_\Delta}}\ar[dr]^{\Inf_{G_\Delta}}& \\ \Aut(\Db(X))_G & &\Aut(\Db_G(X))=\Aut(\Db(Y)).}
\]
Observe that $\Forg_{G_\Delta}$ and $\Inf_{G_\Delta}$ can be thought of as group homomorphisms, since $\KE(\Db_G(X))$ has a natural group structure given by the composition of Fourier--Mukai kernels (see \cite[Sect.\ 2]{Pl}). This yields a natural surjective homomorphism
\[
\mathrm{Lift}:\Aut(\Db_G(X))\lto\Aut(\Db(X))_G/G.
\]
which, composed with the natural map $\Aut(\Db(X))_G/G\to\OO(\widetilde H(X,\ZZ))_G/G$ (see \cite{Or1} or \cite[Prop.\ 10.10]{Hu1}), gives the desired morphism $\Pi$. Notice that an easy computation shows that the morphism $\mathrm{Lift}$ is $2:1$ and the kernel is the group $\langle(-)\otimes\omega_Y\rangle$. In particular we have an isomorphism
\begin{equation}\label{eqn:iso1}
\Aut(\Db(Y))/\langle(-)\otimes\omega_Y\rangle\cong\Aut(\Db(X))_G/G.
\end{equation}

We first prove that $\OO_+(\widetilde H(X,\ZZ))_G/G$ is a subgroup of the image of $\Pi$ of index at most $2$. To this end define $j$
to be the Hodge isometry
$j:=-\id_{H^2(X,\ZZ)}\oplus\id_{H^0(X,\ZZ)\oplus H^4(X,\ZZ)}$ and
take an isometry $\psi\in\OO(\widetilde H (X,\ZZ))_G$. We want to
show that (up to composing $\psi$ with $j$), there exists
$\Phi\in\Aut(\Db(Y))$ such that $\Pi(\Phi)=\psi$. The proof is
divided up in many steps and, besides some key parts which will be
stressed, it is very similar to the one of the the Derived Torelli
Theorem for K3 surfaces (see, for example, \cite[Cor.\
10.12]{Hu1}).

Given an equivalence $\Phi\in\Aut(\Db(X))$, we write $\Phi^H$ for the natural Hodge isometry induced on the Mukai lattice (\cite[Ch.\ 10]{Hu1}).

\bigskip

\noindent{\sc Case 1.} Suppose that $\psi(0,0,1)=\pm (0,0,1)$.

\medskip

Under this hypothesis $\pm \psi(1,0,0)=(r,\ell,s)=:v$ with $r=1$. Hence $v=\exp(\mathrm{c}_1(L))$ for some line bundle $L$ on $X$. If $\psi(0,0,1)=(0,0,1)$ take the  equivalence $\Phi=(-)\otimes L\dual$ (if $\psi(0,0,1)=-(0,0,1)$ just consider $\Phi=(-)\otimes L\dual[1]$). Since $\psi$ is $G$-equivariant, $\iota^*v=v$ and thus $L$ and $\Phi$ are $G$-equivariant as well. The composition $\varphi=\Phi^H\circ\psi$ has the property that $\varphi|_{H^0(X,\ZZ)\oplus H^4(X,\ZZ)}=\id_{H^0(X,\ZZ)\oplus H^4(X,\ZZ)}$.

Since $X$ is the universal cover of $Y$, there exists an ample line bundle $L'\in\Pic(X)$ such that $\iota^*L'\cong L'$. By \cite[Prop.\ VIII.21.1]{BPV}, there are $(-2)$-curves $C_1,\ldots, C_r$ such that
\begin{itemize}
\item[(A.1)] for any $i\in\{1,\ldots,r\}$, $(C_i,\iota^*C_i)=0$ (by $(-,-)$ we denote the cup-product on $H^2(X,\ZZ)$);
\item[(A.2)] the isometry $w=s_{C_1}\circ s_{\iota^*C_1}\circ\cdots\circ s_{C_r}\circ s_{\iota^*C_r}$ is such that $w(\varphi(L'))=\pm L'$.
\end{itemize}
Here $s_{C_i}\in\OO(\widetilde H(X,\ZZ))$ stands for the reflection on cohomology with respect to the vector $(0,[C_i],0)$ (see \cite[Sect.\ VIII.21]{BPV}). Moreover, by its very definition, $\iota^*\circ w=w\circ\iota^*$. Hence, by the Torelli Theorem for K3 surfaces, there is an isomorphism $f:X\isomor X$ with $(j\circ)f_*=w\circ\varphi$ and $\iota^*\circ f_*=f_*\circ\iota^*$.

Take now the spherical twists $T_{\ko_{C_i}(-1)}$. Since $T^H_{\ko_{C_i}(-1)}=s_{C_i}$ and, by \cite[Thm.\ 1.2]{ST}, $T_{\ko_{C_i}(-1)}\circ T_{\ko_{\iota^*C_i}(-1)}=T_{\ko_{\iota^*C_i}(-1)}\circ T_{\ko_{C_i}(-1)}$, the equivalences $\iota^*\circ T_{\ko_{C_i}(-1)}\circ T_{\ko_{\iota^*C_i}(-1)}$ and $T_{\ko_{C_i}(-1)}\circ T_{\ko_{\iota^*C_i}(-1)}\circ\iota^*$ are isomorphic. Hence the equivalence
\[
\Psi:=(T_{\ko_{C_1}(-1)}\circ
T_{\ko_{\iota^*C_1}(-1)}\circ\cdots\circ T_{\ko_{C_r}(-1)}\circ
T_{\ko_{\iota^*C_r}(-1)}\circ\Phi)^{-1}\circ f_*
\]
is such that $\Psi^H=(j\circ)\psi$ and
$\iota^*\circ\Psi=\Psi\circ\iota^*$. Hence there exists
$\Psi'\in\Aut(\Db(Y))$ such that $\mathrm{Lift}(\Psi')=\Psi$ and
$\Pi(\Psi')=(j\circ)\psi$.

\bigskip

\noindent {\sc Case 2.} Suppose now that $\psi(0,0,1)=(r,\ell,s)=:v$, with $r\neq 0$.

\medskip

Up to composing with $-\id_{\widetilde H(X,\ZZ)}$ (which, on the
level of derived categories, just corresponds to the shift $[1]$),
we may assume $r>0$. Then consider the moduli space $M:=M_h(v)$ of
stable (with respect to a generic ample polarization $h$) sheaves
with Mukai vector $v$. The general theory of moduli spaces of
stable sheaves on K3 surfaces in \cite{Mu} ensures that, being $v$
primitive, $M$ is a K3 surface itself (see, for example, the
discussion in \cite[Ch.\ 10]{Hu1}). Moreover, the universal family
$\ke$ on $M\times X$ induces an equivalence
$\Phi_\ke:\Db(M)\isomor\Db(X)$.

\bigskip

\noindent{\sc Claim.} \emph{The involution $\iota:X\to X$ and the equivalence $\Phi_\ke$ induce a fixed-point-free involution $\tilde\iota$ on $M$ such that $\ke$ is equivariant with respect to the group $\widetilde G=\langle\tilde\iota\times\iota\rangle$.}

\medskip

\begin{proof} Consider the commutative diagram
\[
\xymatrix{\Db(M)\ar[d]_{\Phi_\ke}\ar[rr]^{\Psi}& & \Db(M)\ar[d]^{\Phi_\ke}\\ \Db(X)\ar[rr]^{\iota^*}& & \Db(X),}
\]
where $\Psi$ is the natural composition. For any $x\in M$, $\iota^*\ke_x$ is stable and $v(\iota^*\ke_x)=v(\ke_x)$ ($v$ is invariant under the action of $\iota^*$). Hence $\Psi(\ko_x)=\Phi^{-1}_\ke(\iota^*(\Phi_\ke(\ko_y)))=\Phi^{-1}_\ke(\iota^*\ke_x)=\ko_y$, for some closed point $y\in M$. By \cite[Cor.\ 5.23]{Hu1}, $\Psi=((-)\otimes L)\circ f_*$, for some line bundle $L\in\Pic(X)$ and an isomorphism $f:M\isomor X$.

Observe that $\Psi^H$ is not the identity. Indeed,
$\Psi^H|_{T(M)}=((\Phi_\ke^H)^{-1}\circ\iota^*|_{\Phi^H_\ke(T(M))}\circ\Phi_\ke^H)|_{T(M)}$.
By \cite[Thm.\ 1.5]{Mu},
$\Phi^H_\ke(T(M))=T(X)\subseteq\Lambda_-$, where $\Lambda_-$ is
the eigenspace of the eigenvalue $-1$ of $\iota^*$. Hence
$\Psi^H|_{T(M)}=-\id_{T(M)}$.

Take now the commutative diagram
\begin{eqnarray}\label{eqn:cd1}
\xymatrix{\widetilde H(M,\ZZ)\ar[d]_{\Phi_\ke^H}\ar[rr]^{\Psi^H}& & \widetilde H(M,\ZZ)\ar[d]^{\Phi_\ke^H}\\ \widetilde H(X,\ZZ)\ar[rr]^{\iota^*}\ar[d]_{\psi}& & \widetilde H(X,\ZZ)\ar[d]^{\psi}\\\widetilde H(X,\ZZ)\ar[d]_{\Psi_1^H}\ar[rr]^{\iota^*}& & \widetilde H(X,\ZZ)\ar[d]^{\Psi_1^H}\\ \widetilde H(X,\ZZ)\ar[rr]^{\iota^*}& & \widetilde H(X,\ZZ),}
\end{eqnarray}
where $\Psi_1^H$ is the isometry induced on cohomology by an equivalence $\Psi_1:\Db(X)\isomor\Db(X)$  with the following properties:
\begin{itemize}
\item[(B.1)] it is the composition of the tensorization with a(n equivariant) line bundle on $X$ and the composition of spherical twists $T_{\ko_{C_i}(-1)}\circ T_{\iota^*\ko_{C_i}(-1)}$, where $C_1,\ldots,C_r$ are rational curves on the K3 surface $X$ as in (A.1);
\item[(B.2)] there exists an isomorphism $f_1:M\isomor X$ such that $(j\circ)(f_1)_*=\Psi_1^H\circ \psi\circ\Phi_\ke^H$.
\end{itemize}
Notice that such a functor exists due to the fact that $\psi$ is equivariant.

Diagram \eqref{eqn:cd1} can be rewritten on the level of derived categories in the following way
\[
\xymatrix{\Db(M)\ar[d]_{(f_1)_*}\ar[rr]^{\Psi}& & \Db(M)\ar[d]^{(f_1)_*}\\ \Db(X)\ar[rr]^{\iota^*}& & \Db(X).}
\]
This proves that $\Psi=f_*$. Moreover, $f$ has no fixed points. Indeed, suppose that for, some closed point $x\in M$,
\[
\ko_x=\Psi(\ko_x)=f^*(\iota^*(f_*(\ko_x))).
\]
Then $f_*(\ko_x)=\iota^*(f_*(\ko_x))$. But $\iota^*$ does not fix any skyscraper sheaf by definition. This concludes the proof of the claim.\end{proof}

Now put $\Phi=\Phi_\ke$ if $r>0$ and $\Phi=\Phi_{\ke[1]}$ if $r<0$. By the previous claim, $M$ has a fixed-point-free involution $\tilde\iota$ making the equivalence $\Phi$ equivariant with respect to the group $\widetilde G$. Therefore the composition $\psi':=(\Phi^{-1})^H\circ\psi$ is such that $\psi'(0,0,1)=\pm(0,0,1)$ and one proceeds as in Case 1.

\bigskip

\noindent {\sc Case 3.} Finally assume that $\psi(0,0,1)=(0,\ell,s)=:v$, with $\ell\neq 0$.

\medskip

Since $\iota^* v=v$, take a $G$-equivariant line bundle $L'\in\Pic(X)$ such that $(L',\ell)+s\neq 0$. The equivalence $T_{\ko_X}$ is $G$-equivariant. Thus the composition $(T_{\ko_X}\circ((-)\otimes L'))^H\circ\psi$ is as in Case 2 and we proceed as before.

\bigskip

To prove that $\OO_+(\widetilde H(X,\ZZ))_G/G$ is
a subgroup of the image of $\Pi$ of index at most $2$, it is
enough to show that all the equivalences involved in the above
constructions are orientation preserving. But for this we can just
apply \cite[Sect.\ 5]{HS1}.

To conclude the proof, we need to observe that $\OO_+(\widetilde
H(X,\ZZ))$ is the image of the natural morphism
$\Aut(\Db(X))\to\OO(\widetilde H(X,\ZZ))$. This is the content of \cite[Corollary 3]{HMS1}.
\end{proof}

Going back to Proporision \ref{prop:BM} and hence to the case of
distinct Enriques surfaces, the same proof yields the following:

\begin{cor}\label{cor:dtt}
Let $Y_1$ and $Y_2$ be Enriques surfaces and let $X_1$ and $X_2$
be the universal covers endowed with the involutions $\iota_1$ and
$\iota_2$.

{\rm (i)} Any equivalence $\Db(Y_1)\cong\Db(Y_2)$ induces an
orientation preserving equivariant Hodge isometry $\widetilde
H(X_1,\ZZ)\cong\widetilde H(X_2,\ZZ)$.

{\rm (ii)} Any orientation preserving equivariant Hodge isometry
$\widetilde H(X_1,\ZZ)\cong\widetilde H(X_2,\ZZ)$ lifts to an
equivalence $\Db(Y_1)\cong\Db(Y_2)$.
\end{cor}

If we do not allow to use the result in \cite{HMS1}, Proposition
\ref{thm:DTT} has to be weakened in the following way:

\begin{prop}
There exists a natural morphism of groups
$\Pi:\Aut(\Db(Y))\lto\OO(\widetilde H(X,\ZZ))_G/G$ whose image
contains the index-2 subgroup $\OO_+(\widetilde H(X,\ZZ))_G/G$.
\end{prop}

Following \cite{B2}, given an Enriques surface $Y$ and its universal cover $X$, define the open subset $\kp(X)\subseteq\kn(X)\otimes\CC$ consisting of those vectors whose real and imaginary parts span a positive definite two plane in $\kn(X)\otimes\RR$. Denote by $\kp^+(X)$ one of the two connected components of $\kp(X)$. If $\Delta(X)$ is the set of vectors in $\kn(X)$ with self-intersection $-2$, Bridgeland considers
\[
\kp_0^+(X):=\kp^+(X)\setminus\bigcup_{\delta\in\Delta(X)}\delta^\perp.
\]
Define $\kp^+_0(Y):=\Forg_{G*}\dual(\kp_0^+(X))_G$ and take the group $\Aut^0(\Db(Y))$ (respectively $\Aut^0(\Db(X))$) of those autoequivalences preserving $\Sigma(Y)$ (respectively $\NDStab(\Db(X))$) and inducing the identity on cohomology via the morphisms $\Pi$.

\begin{cor}\label{cor:cov} The map $\kz:\Sigma(Y)\to\kn(Y)\otimes\CC$ in Proposition \ref{prop:connEnr} defines a covering map onto $\kp_0^+(Y)$ such that $\Aut^0(\Db(Y))/\langle (-)\otimes\omega_Y\rangle$ acts as the group of deck transformations.\end{cor}

\begin{proof} The fact that $\kz$ is a covering map is an easy consequence of the commutativity of diagram \eqref{eqn:commdiagr} in Proposition \ref{prop:connEnr}.

Let $\Phi\in\Aut^0(\Db(X))$ be such that $\Phi(\Gamma_X\cap\NDStab(\Db(X)))=\Gamma_X\cap\NDStab(\Db(X))$. By definition, $\Phi(\sigma)=\iota^*(\Phi(\sigma))=\Phi'(\iota^*(\sigma))=\Phi'(\sigma)$, for some autoequivalence $\Phi'\in\Aut^0(\Db(X))$ and for all $\sigma\in\Gamma_X\cap\NDStab(\Db(X))$. Due to \cite[Thm.\ 1.1]{B2}, this is enough to conclude that $\Phi=\Phi'$ and $\Phi\in\Aut^0(\Db(X))_G$.

Since $\Aut^0(\Db(X))_G$ injects into $\Aut(\Db(X))_G/G$, the commutativity of diagram \eqref{eqn:commdiagr} in Proposition \ref{prop:connEnr} and the isomorphism \eqref{eqn:iso1} allow us to conclude that $\Aut^0(\Db(Y))/\langle (-)\otimes\omega_Y\rangle$ acts as the group of deck transformations.\end{proof}

Therefore we can state the following conjecture generalizing \cite[Conj.\ 1.2]{B2} to the case of Enriques surfaces.

\begin{conjecture}\label{conj:BriEnr} The group $\Aut(\Db(Y))$ preserves $\Sigma(Y)$ and, moreover, $\Sigma(Y)$ is connected and simply connected.\end{conjecture}

From the previous conjecture would follow a complete description
of the kernel of the morphism $\Pi$ in Proposition \ref{thm:DTT}.
In particular, we would get the existence of the short exact
sequence
\[
1\lto\pi_1(\kp_0^+(Y))\lto\Aut(\Db(Y))/\langle(-)\otimes\omega_Y\rangle\lto\OO_+(\widetilde H(X,\ZZ))_G/G\lto 1
\]
and the fact that $\Sigma(Y)=\NStab^\dag(\Db(Y))$. In the next section we will prove (Proposition \ref{prop:conngen}) that $\Sigma(Y)$ is connected when $Y$ is a generic Enriques surface.

\bigskip

\subsection{Spherical objects and generic Enriques surfaces}\label{subsec:GeomEnr}~

\medskip

An Enriques surface $Y$ with universal cover $X$ is \emph{generic} if the rank of the Picard group $\Pic(X)$ is $10$. Due to the main results in \cite{Ho1,Ho2}, $10$ is the minimal possible Picard number of a K3 surface which is the universal cover of an Enriques surface. Moreover, the set of all generic Enriques surfaces is dense in the moduli space of such surfaces.

\begin{lem}\label{lem:inv} Let $Y$ be a generic Enriques surface. Then
the involution $\iota:X\to X$ on the universal cover of $Y$ is
such that $\iota^*|_{\Pic(X)}=\id_{\Pic(X)}$,
$\iota^*|_{T(X)}=-\id_{T(X)}$ and
$\Pic(X)\cong\Lambda_+:=U(2)\oplus E_8(-2)$.\end{lem}

\begin{proof} By \cite[Thm.\ 5.1]{Ho2}, for the K3 surface $X$ there
exist isometries $\psi:H^2(X,\ZZ)\isomor\Lambda:=U^{\oplus
3}\oplus E_8(-1)^{\oplus 2}$ and $\theta:\Lambda\isomor\Lambda$
such that the following diagram commutes
\begin{eqnarray}\label{eqn:inv}
\xymatrix{H^2(X,\ZZ)\ar[d]_{\psi}\ar[rr]^{\iota^*}& &
H^2(X,\ZZ)\ar[d]^{\psi}\\ \Lambda\ar[rr]^{\theta}& & \Lambda.}
\end{eqnarray}
Denote by $\Lambda_+\hookrightarrow\Lambda$ the eigenspace of the
eigenvalue $+1$ of $\theta$. By \cite{Ho1,Ho2},
$\Lambda_+=U(2)\oplus E_8(-2)$ (for a definition of the lattices
$U$ and $E_8(-1)$ appearing in the previous discussion, see for
example \cite[Ch.\ I]{BPV}) and
$\psi^{-1}(\Lambda_+)\subseteq\Pic(X)$. Since $\Lambda_+$ has rank
$10$ and $Y$ is generic, this concludes the proof.
\end{proof}

\begin{remark}\label{rmk:nocurves} In the generic case, by Lemma \ref{lem:inv},
$v^2$ is divisible by $4$, for any $v\in\Pic(X)$. Hence $X$ does
not contain rational curves. Since a rational curve in $Y$ lifts
to a pair of (disjoint) rational curves in $X$, $Y$ does not
contain rational curves either.\end{remark}

We can now say more about the set $\Sigma(Y)$ defined in Proposition \ref{prop:connEnr}.

\begin{prop}\label{prop:conngen} Let $Y$ be a generic Enriques surface. Then $\NDStab(\Db(X))\subseteq\NStab(\Db(X))$ is isomorphic to $\Sigma(Y)$, which is then connected. In particular $\Sigma(Y)=\NStab^\dagger(\Db(Y))$.\end{prop}

\begin{proof}  By Lemma \ref{lem:inv}, $\kp^+_0(Y)=\kp^+_0(X)_G=\kp^+_0(X)$. Hence $\Gamma_X$ is an open and closed subset in the connected component $\NDStab(\Db(X))$ of the same dimension.\end{proof}

In this section we will also be interested in characterizing special objects in $\Db(Y)$. In particular, recall the following definition:

\begin{definition}\label{def:sohrig} Let $Z$ be a smooth projective variety with canonical bundle $\omega_Z$.

(i) An object $\ke\in\Db(Z)$ is \emph{strongly rigid} if
$(\ke,\ke)^1=0$ and $\ke\cong\ke\otimes\omega_Z$;

(ii)  An object $\ke\in\Db(Z)$ is \emph{spherical} if $(\ke,\ke)^i=1$ for $i\in\{0,2\}$ and otherwise zero and $\ke\cong\ke\otimes\omega_Z$.\end{definition}

For a variety $Z$ we denote by $\Sph(Z)$ and $\Rig(Z)$ the sets of
spherical and strongly rigid objects in $\Db(Z)$ respectively.


\begin{lem}\label{lem:rigupst} Let $\pi:X\to Y$ be as before. If $\ke\in\Db(Y)$ is strongly rigid,
then there exists $\kf\in\Db(X)$ strongly rigid and such that
$\ke=\pi_*(\kf)$.\end{lem}

\begin{proof} Since $\ke$ is strongly rigid, $\ke\cong\ke\otimes\omega_Y$.
By \cite[Lemma 7.16]{Hu1}, there exists $\kf\in\Db(X)$ with
$\ke=\pi_*\kf$. The observation that
$\pi^*\pi_*\ke=\ke\oplus\iota^*\ke$ yields  the following list of
isomorphisms:
\begin{equation*}\begin{split}
\Hom_{\Db(Y)}^i(\ke,\ke)&\cong\Hom_{\Db(Y)}^i(\pi_*\kf,\pi_*\kf)\\&\cong\Hom_{\Db(X)}^i(\pi^*\pi_*\kf,\kf)\\&\cong\Hom_{\Db(X)}^i(\kf\oplus\iota^*\kf,\kf)\\&\cong\Hom_{\Db(X)}^i(\kf,\kf)\oplus\Hom_{\Db(X)}^i(\iota^*\kf,\kf).
\end{split}\end{equation*}
Therefore if $\ke$ is strongly rigid, $\kf$ is strongly rigid as
well.\end{proof}

\begin{lem}\label{lem:Mukai} Consider in $\Db(Y)$ a triangle
\[
\ka\lto\ke\lto\kb\lto\ka[1]
\]
such that $(\ka,\kb)^r=(\kb,\kb)^s=0$, for $r\leq 0$ and $s<0$. Assume moreover that $\ke\cong\ke\otimes\omega_Y$ and $\Hom_{\Db(Y)}^{\leq 0}(\ka,\kb\otimes\omega_Y)=0$. Then $\ka\cong\ka\otimes\omega_Y$, $\kb\cong\kb\otimes\omega_Y$ and $(\ka,\ka)^1+(\kb,\kb)^1\leq (\ke,\ke)^1$.
\end{lem}

\begin{proof} Let $f:\ke\to\ke\otimes\omega_Y$ be the isomorphism in the hypotheses. Clearly $\Hom_{\Db(Y)}(\ka,\kb)=0$ and, since $\omega_Y^{\otimes 2}=\ko_Y$, $\Hom_{\Db(Y)}(\ka,\kb\otimes\omega_Y)\cong\Hom_{\Db(Y)}(\ka\otimes\omega_Y,\kb)=0$. Thus there exist morphisms $f_1$, $f_2$, $g_1$ and $g_2$ making the following diagram commutative:
\begin{eqnarray}\label{eqn:dia1}
\xymatrix@1{\ka\ar[d]_{f_1}\ar[r]\ar@/_2pc/[dd]_{h_1} &
\ke\ar[d]_{f}\ar[r]\ar@/^2pc/[dd]|\hole^(.3){\id}&
\kb\ar[d]_{f_2}\ar@/^2pc/[dd]^{h_2}\\
\ka\otimes\omega_Y\ar[d]_{g_1}\ar[r]&
\ke\otimes\omega_Y\ar[d]_{f^{-1}}\ar[r]&
\kb\otimes\omega_Y\ar[d]_{g_2}\\
\ka\ar[r]&\ke\ar[r]&\kb.}
\end{eqnarray}
Since $\Hom^{-1}_{\Db(Y)}(\ka,\kb)=0$, the morphisms $h_1$ and $h_2$ are uniquely determined and so they must be the identity. Repeating the same argument, starting from the triangle $\ka\otimes\omega_Y\lto\ke\otimes\omega_Y\lto\kb\otimes\omega_Y\lto\ka\otimes\omega_Y[1]$, one immediately concludes that $f_1$ and $f_2$ are isomorphisms.

The fact that $(\ka,\ka)^1+(\kb,\kb)^1\leq (\ke,\ke)^1$ is now obtained repeating the proof of \cite[Lemma 2.7]{HMS}, simply using Serre duality and the isomorphisms $f_1$ and $f_2$ previously defined.\end{proof}

In the specific case of Enriques surfaces we will need a third class of objects in the derived category.

\begin{definition}\label{def:quasisph} An object $\ke\in\Db(Y)$ is \emph{quasi-spherical} if $\ke\cong\kf\oplus\kf\otimes\omega_Y$ where $\pi^*\kf$ is a spherical object.\end{definition}

Notice that, by definition, $\pi^*\kf$ in the previous definition is automatically $G$-invariant. We now complete the proof of the last statement in Theorem \ref{thm:main1}.

\begin{prop}\label{prop:rigid} Let $Y$ be a generic Enriques surface. Then $\Sph(Y)$ is empty, while
$\Rig(Y)$ consists of objects which are extensions of quasi-spherical objects.\end{prop}

\begin{proof} Let $\ke\in\Db(Y)$ be spherical. Due to Lemma \ref{lem:rigupst}, there exists $\kg\in\Rig(X)$ such that $\ke=\pi_*\kg$. Then, by the argument in the proof of Lemma \ref{lem:rigupst},
\[
\Hom^i_{\Db(Y)}(\ke,\ke)\cong\Hom^i_{\Db(X)}(\kg,\kg)\oplus\Hom^i_{\Db(X)}(\iota^*\kg,\kg).
\]
This implies that $\chi(\ke,\ke)=\chi(\kg,\kg)+\chi(\iota^*\kg,\kg)=2\chi(\kg,\kg)=4$. The penultimate equality is due to the fact that, being $Y$ generic, $\iota^*$ acts as the identity on $H^0(X,\ZZ)\oplus\NS(X)\oplus H^4(X,\ZZ)$ (Lemma \ref{lem:inv}). In particular, $\ke$ is not a $(-2)$-class.

Suppose now that $\ke'\in\coh(Y)$ is strongly rigid and consider
the short exact sequence
\[
0\lto\ke'_\mathrm{tor}\lto\ke'\lto\kf\lto 0,
\]
where $\ke'_\mathrm{tor}$ is the torsion part of $\ke'$ and $\kf$
is torsion free. Since, clearly,
$$\Hom_{\Db(Y)}(\ke'_\mathrm{tor},\kf)=\Hom_{\Db(Y)}(\ke'_\mathrm{tor},\kf\otimes\omega_Y)=0,$$
we can now apply Lemma \ref{lem:Mukai}, concluding that
$\ke'_\mathrm{tor}$ and $\kf$ are both strongly rigid objects.
Since $Y$ is generic, by Remark \ref{rmk:nocurves}, it does not
contain rational curves and $\ke'_{\mathrm{tor}}$ must be
supported on points. On the other hand, a very easy computation
shows that there are no strongly rigid objects in $\coh(Y)$
supported on points and so $\ke'_\mathrm{tor}=0$.

Fix an ample polarization $\ell$ on $Y$ and consider the
HN-filtration of $\kf$ and the exact sequence
\[
0\lto\kf_1\lto\kf\lto\kf_2\lto 0,
\]
where $\kf_2$ is the first $\mu_\ell$-semistable factor in the
filtration. Since $\omega_Y$ is a torsion class,
$\mu_\ell^-(\kf_1)>\mu_\ell(\kf_2)=\mu_\ell(\kf_2\otimes\omega_Y)$
(here $\mu_\ell^-(\kf_1)$ is the slope of the last
$\mu_\ell$-semistable factor in the HN-filtration of $\kf_1$).
Therefore,
$\Hom_{\Db(Y)}(\kf_1,\kf_2)=\Hom_{\Db(Y)}(\kf_1,\kf_2\otimes\omega_Y)=0$
and Lemma \ref{lem:Mukai} applies allowing us to conclude that
$\kf_1$ and $\kf_2$ are strongly rigid as well. Thus, by induction
on the length of the HN-filtration, we can assume that $\kf$ is
$\mu_\ell$-semistable.

Notice that if $\ke\in\coh(Y)$ is $\mu_\ell$-stable then
$\ke\otimes\omega_Y$ is $\mu_\ell$-stable as well. Hence there are
two possibilities: Either $\kf$ has at most two $\mu_\ell$-stable
factors $\ke$ and $\ke\otimes\omega_Y$ or there are
$\ke_1,\ke_2\in\coh(Y)$ which are $\mu_\ell$-stable factors of
$\kf$ but $\ke_1\not\cong\ke_2$ and
$\ke_1\not\cong\ke_2\otimes\omega_Y$.

\medskip

If we are in the first case, since
$\chi(\ke,\ke)=\chi(\ke,\ke\otimes\omega_Y)=\chi(\ke\otimes\omega_Y,\ke\otimes\omega_Y)$,
the following holds true
\[
0<\chi(\kf,\kf)=k^2\chi(\ke,\ke),
\]
where $k$ is the number of $\mu_\ell$-stable factors of $\kf$.

Since $\ke$ is stable and
$\mu_\ell(\ke)=\mu_\ell(\ke\otimes\omega_Y)$, if
$\Hom^2_{\Db(Y)}(\ke,\ke)\neq 0$, then
$\ke\cong\ke\otimes\omega_Y$. Hence
$\chi(\ke,\ke)=2-(\ke,\ke)^1>0$ implies $(\ke,\ke)^1=0$ (notice
that $(\ke,\ke)^1$ is even in this case). In particular $\ke$
should be spherical, which is impossible.

Suppose now that $\Hom_{\Db(Y)}^2(\ke,\ke)=0$. Obviously,
$\chi(\ke,\ke)=1-(\ke,\ke)^1>0$ and so $(\ke,\ke)^1=0$. Thus
$\kf\cong\ke^{\oplus m}\oplus(\ke\otimes\omega_Y)^{\oplus m}$, for
some positive integer $m$.  Let us show that $\ke$ is locally
free. Indeed, it is enough to apply Lemma \ref{lem:Mukai} to the
triangle
\[
\kt[-1]\lto\kf\lto\kf^{\vee\vee}\lto\kt,
\]
where $\kt$ is a torsion sheaf supported on points. Since
$$\Hom_{\Db(Y)}^{\leq
0}(\kt[-1],\kf^{\vee\vee})=\Hom_{\Db(Y)}^{\leq
0}(\kt[-1],\kf^{\vee\vee}\otimes\omega_Y)=0,$$ Lemma
\ref{lem:Mukai} implies once more that $\kt$ should be strongly
rigid which is a contradiction unless $\kt=0$. Hence $\kf$ is
locally free. Observe that $\pi^*\ke$ is spherical. Indeed, by
adjunction,
$\Hom_{\Db(X)}^i(\pi^*\ke,\pi^*\ke)=\Hom_{\Db(Y)}^i(\ke,\pi_*\pi^*\ke)=\Hom_{\Db(Y)}^i(\ke,\ke\oplus\ke\otimes\omega_Y)$
and thus, by Serre duality,
\[
\Hom_{\Db(X)}^i(\pi^*\ke,\pi^*\ke)\cong\left\lbrace\begin{array}{ll}\Hom_{\Db(Y)}(\ke,\ke)&\mbox{if
}i\in\{0,2\}\\0&\mbox{otherwise}.\end{array}\right.
\]
This completes the case of $\kf$ with at most two stable factors $\ke$ and
$\ke\otimes\omega_Y$.

\medskip

Assume that there are two sheaves $\ke_1$ and $\ke_2$ which are
$\mu_\ell$-stable factors of $\kf$ but $\ke_1\not\cong\ke_2$ and
$\ke_1\not\cong\ke_2\otimes\omega_Y$. We want to show that there
exists a short exact sequence
\begin{eqnarray}\label{eqn:1}
0\lto\kg_1\lto\kf\lto\kg_2\lto 0,
\end{eqnarray}
with $\kg_2$ which is extension only of $\ke$ and
$\ke\otimes\omega_Y$, for some $\mu_\ell$-stable sheaf $\ke$, and
$\Hom_{\Db(Y)}(\kg_1,\ke)=\Hom_{\Db(Y)}(\kg_1,\ke\otimes\omega_Y)=0$.

In this case, by construction,
$\Hom_{\Db(Y)}(\kg_1,\kg_2)=\Hom_{\Db(Y)}(\kg_1,\kg_2\otimes\omega_Y)=0$
and so, applying Lemma \ref{lem:Mukai} to \eqref{eqn:1}, we
conclude that $\kg_1$ and $\kg_2$ are both strongly rigid, $\kg_2$
is as in the previous case and the number of $\mu_\ell$-stable
factors of $\kg_1$ is smaller than the one of $\kf$. Hence, one
proceed recursively analyzing further $\kg_1$. Since the
JH-filtration of $\kf$ is finite, the process terminates in a
finite number of steps.

To produce the short exact sequence \eqref{eqn:1}, take a
$\mu_\ell$-stable factor $\ke$ of $\kf$ with a morphism
$\kf\to\ke$. Then there exist $\kf_1$ and $\kf_2$ such that
$\kf_2$ is extension of $\ke$, they fit in the short exact
sequence
\[
0\lto\kf_1\lto\kf\lto\kf_2\lto 0,
\]
and $\Hom_{\Db(Y)}(\kf_1,\kf_2)=\Hom_{\Db(Y)}(\kf_1,\ke)=0$. If
$\Hom_{\Db(Y)}(\kf_1,\ke\otimes\omega_Y)=0$, then set
$\kg_1:=\kf_1$ and $\kg_2:=\kf_2$. Otherwise, if
$\Hom_{\Db(Y)}(\kf_1,\ke\otimes\omega_Y)\neq 0$, then, as before,
there exist $\kf_3$, $\kf_4$ and a short exact sequence
\[
0\lto\kf_3\lto\kf_1\lto\kf_4\lto 0,
\]
where $\kf_4$ is extension of $\ke\otimes\omega_Y$ and
$\Hom_{\Db(Y)}(\kf_3,\ke\otimes\omega_Y)=0$. If
$\Hom_{\Db(Y)}(\kf_3,\ke)\neq 0$, we continue filtering until we
get \eqref{eqn:1}. Again, the process terminates because the
JH-filtrations are finite.

\medskip

Consider a strongly rigid complex $\ke\in\Rig(Y)$ and let $N$ be
the maximal integer such that $\kh^N(\ke)\neq 0$. Hence we have
the following triangle
\[
\kf\lto\ke\lto\kh^N(\ke)[-N]
\]
with $\Hom_{\Db(Y)}^{\leq
0}(\kf,\kh^N(\ke)[-N])=\Hom_{\Db(Y)}^{\leq
0}(\kf,\kh^N(\ke)\otimes\omega_Y[-N])=0$. Lemma \ref{lem:Mukai}
implies that $\kh^N(\ke)$ is strongly rigid and hence locally
free. The same happens for the other cohomology sheaves,
proceeding by induction on $\kf$.\end{proof}

\begin{remark}\label{rmk:K3Enriques} Generic Enriques surfaces provide examples
of smooth surfaces with no spherical objects but plenty of
(strongly) rigid objects. The absence of spherical objects for
those surfaces is in marked contrast with the case of their
closest relatives: K3 surfaces. Indeed, in that case, spherical
objects are always present (at least in the untwisted case). As
was proved in \cite{HMS}, the only way to reduce drastically the
number of (strongly) rigid and spherical objects is to pass to
twisted or generic analytic K3 surfaces.
\end{remark}

\section{Local Calabi--Yau varieties}\label{sec:CY}

We consider some further situations in which the techniques of
Section \ref{subsec:stabex} can be applied. In Section
\ref{subsec:OmegaPN} we compare the spaces of stability conditions
on projective spaces $\PP^N$and those on their canonical bundles
$|\omega_{\PP^N}|$. In this case the relation is slightly weaker
than before and we obtain only a map between particular open
subsets. For $N=1$, where the two stability manifolds can be
completely described, the open subset of $\Stab(\PP^1)$ in
question corresponds to the ``non-degenerate'' stability
conditions, while the one of  $\Stab(|\omega_{\PP^2}|)$ to a
fundamental domain with respect to the action of the group of
autoequivalences.

In Section \ref{subsec:moreex} we examine spaces of stability
conditions for resolutions of Kleinian singularities (with
particular attention to the case of $A_2$-singularities) and
compare them with the spaces of stability conditions of the
corresponding quivers. We then conclude by studying the case of
local K3 surfaces (analytically) embedded into projective ones.
Here the relation is even weaker: only few stability conditions on
projective K3 surfaces induce stability conditions on the local
ones. This can be intuitively understood by thinking of the space
of stability conditions on the canonical bundle over the
projective line as a sort of ``limit'' of the stability manifold
for projective K3 surfaces.

\bigskip

\subsection{Stability conditions on canonical bundles on projective spaces}\label{subsec:OmegaPN}~

\medskip

We start by recalling briefly the notion of mutation of
exceptional objects in $\Db(\PP^N)$. Let $(\ke,\kf)$ be a strong
exceptional pair. We define the objects $\mc{L}_\ke \kf$ and
$\mc{R}_\kf \ke$ (which we call \emph{left mutation} and
\emph{right mutation} respectively) by means of the distinguished
triangles
$$\mc{L}_\ke \kf\to\Hom_{\Db(\PP^N)}(\ke,\kf)\otimes\ke\to\kf,$$
$$\ke\to\Hom_{\Db(\PP^N)}(\ke,\kf)^\vee\otimes\kf\to \mc{R}_\kf \ke.$$
A \emph{mutation} of a strong exceptional collection $E=\grf{\ke_0,\ldots,\ke_n}$ is defined as a mutation of a pair of adjacent objects in $E$ and it will be denoted by
\begin{align*}
\mc{R}_i E=& \grf{\ke_0, \ldots ,\ke_{i-1},\ke_{i+1}, \mc{R}_{\ke_{i+1}}\ke_i,\ke_{i+2}, \ldots,\ke_n},\\
\mc{L}_i E =& \grf{\ke_0,\ldots,\ke_{i-1}, \mc{L}_{\ke_i} \ke_{i+1},\ke_i, \ke_{i+2}, \ldots ,\ke_n},
\end{align*}
for $i\in\{0, \ldots , n-1\}$. By a result of Bondal (\cite[Assertion 9.2]{Bond}), applying a mutation to a complete strong exceptional collection consisting of sheaves on $\PP^N$ we get a complete strong exceptional collection consisting of sheaves as well. An \emph{iterated mutation} of a strong exceptional collection is the iterated application of a finite number of mutations.

\medskip

An open subset of the stability manifold of $\Db (\PP^N)$ was studied in \cite{emolo1}.
Let $E=\grf{\ke_0, \ldots,\ke_N}$ be a strong complete exceptional collection on $\Db (\PP^N)$ consisting of sheaves. Define $\Ss_{E} (\PP^N)$ as the union of the open subsets $\Theta_{F}$ (see Example \ref{ex:PN}) over all iterated mutations $F$ of $E$. It is proved in \cite[Cor.\ 3.20]{emolo} that $\Ss_E  (\PP^N) \cont \Stab (\PP^N)$ is an open and connected $(N+1)$-dimensional submanifold. For $N=1$ all strong complete exceptional collections are obtained as iterated mutations of $$O := \grf{\ko_{\PP^1}, \ko_{\PP^1} (1)}$$ and $\Sigma_{O} (\PP^1)$ is equal to the full stability manifold $\Stab (\PP^1)$ (\cite{emolo,Okada1}).

Let $X$ be the total space of the canonical bundle $V := \ww_{\PP^N} \os{\pp}{\to} \PP^N$ and let $i : \PP^N \hookrightarrow X$ denote the zero-section and $C$ its image. Denote by $\Stab (X)$ the stability manifold of $\Db_0 (X) := \Db_C (\coh(X))$, the full triangulated subcategory of $\Db (\coh (X))$ whose objects have cohomology sheaves supported on $C$.

We want to compare the open subset $\Ss_E (\PP^N)$ with the space of stability conditions on $X$.
Consider the exact faithful functor $i_* : \Db (\PP^N) \to \Db_0 (X)$. By d\'{e}vissage, the essential image of $i_*$ generates $\Db_0 (X)$.
Take the abelian category
\begin{equation}\label{eqn:Q}
\cat{Q}_0^{E} := \langle \ke_0 [N], \ke_1 [N-1] \ldots, \ke_N \rangle
\end{equation}
and consider the open set $U_0^{E} \cont \Theta_E$ consisting of those stability conditions $\sigma\in\Theta_E$ of the form $\sigma_0^E\cdot(G,f)$, for $(G,f)\in\glpiu$, and $\sigma_0^E$ having $\cat{Q}_0^{E}$ as heart.
It is easy to see that $\cat{Q}_0^{E}$ is $i_*$-admissible.
Then, by Proposition \ref{pro:existence} (or Theorem \ref{thm:Poli}), we have an open connected subset $\til{U}_0^{E}$ of $\Stab (X)$ such that $i_*^{-1} : \til{U}_0^{E} \isomor U_0^{E}$.
Denote by $\Gg^{E}$ the open subset of $\Stab (X)$ defined as the union of the open subsets $\til{U}_0^{F}$, for $F$ an iterated mutation of $E$. By \cite[Cor.\ 3.20]{emolo} it is connected. By Lemma \ref{prop:2.5}, we have the following result.

\begin{prop}\label{prop:compPN}
The morphism $i_*^{-1} : \Gg^{E} \to \Ss_{E}  (\PP^N)$ is an open embedding.
\end{prop}

\begin{remark}\label{rmk:imagePN}
Notice that the image of the morphism $i_*^{-1}$ consists of those stability conditions in $\Ss_E (\PP^N)$ whose heart, up to the action of $\glpiu$, is \emph{faithful}, i.e.\ its bounded derived category is equivalent to $\Db (\PP^N)$. More precisely, up to the action of $\glpiu$, the heart of a stability condition in the image of $i_*^{-1}$ is equivalent to the abelian category of finitely generated modules over the algebra $\End\left(\bigoplus \kg_i\right)$, for $G = \grf{\kg_0, \ldots, \kg_N}$ an iterated mutation of $E$.
\end{remark}

\begin{remark}\label{rmk:quivery}
The abelian categories $\Ff (\langle i_* \cat{Q}_0^{E} \rangle)$,
for a Fourier--Mukai autoequivalence $\Ff$ of $\Db_0 (X)$, are
called in \cite{B3} \emph{quivery subcategories}. If $\grf{\ks_0,
\ldots ,\ks_N}$ are the minimal objects of a quivery subcategory,
then there is a canonical cyclic ordering in which
$$\Hom^k_{\Db_0 (X)}(\ks_i,\ks_j) = 0, \ \text{unless} \ 0 \leq k \leq N \ \text{and} \ i - j \equiv k\pmod{N}.$$
An ordered quivery subcategory is a quivery subcategory in which an ordering of its minimal objects is fixed and is compatible with the canonical cyclic ordering. Bridgeland proved \cite[Thm.\ 4.11]{B3} that there is an action of the braid group $B_{N+1}$ (i.e.\ the group generated by elements $\tau_i$, indexed by the cyclic group $\ZZ/(N+1)\ZZ$, together with a single element $r$, subject to the relations $r \tau_i r^{-1} = \tau_{i+1}$, $r^{N+1} = 0$ and, if $N \geq 2$, $\tau_i \tau_{i+1} \tau_i = \tau_{i+1} \tau_i \tau_{i+1}$ and for $j-i \neq \pm 1$, $\tau_i \tau_j = \tau_j \tau_i$\footnote{Notice that the condition $N\geq 2$ is missing in \cite{B3}. Indeed, if $N=1$, $\tau_0\tau_{1}\tau_0\neq\tau_{1} \tau_0\tau_{1}$.}) on the set of ordered quivery subcategories of $\Db_0 (X)$ which essentially corresponds to tilting at minimal objects:
\begin{align*}
r \grf{\ks_0, \ldots ,\ks_N} &:= \grf{\ks_N,\ks_0, \ldots ,\ks_{N-1}}\\
\tau_i \grf{\ks_0, \ldots ,\ks_N} & := \grf{\ks_0, \ldots ,\ks_{i-2},\ks_i [-1], T_{\ks_i}\ks_{i-1},\ks_{i+1}, \ldots, \ks_N},
\end{align*}
with $i\in\{1, \ldots, N\}$. For the lower dimensional cases $N\in\{1, 2\}$ this action is free (see \cite[Thm.\ 5.6]{B3} for $N=2$, whose proof works also for $N=1$).
\end{remark}

Consider the spherical twists $T_{i_*\ke_k}$, for $\ke_k\in E$,
and take the subgroup $G_E$ of $\Aut(\Db_0 (X))$ generated by
these functors, tensorizations by line bundles, automorphisms of
$X$, and shifts. When $N=1$, \cite[Thm.\ 1.3]{IU} or, in a more
intrinsic way, Theorem \ref{thm:stab1}, \cite[Thm.\ 1.1]{B4} and
the classification of the spherical objects in $\Db_0 (X)$ (see
\cite{IU}) yields $G_E=\Aut(\Db_0 (X))$. (Notice that, for $N=1$,
all autoequivalences of $\Db_0 (X)$ are of Fourier--Mukai type, by
\cite[Appendix A]{IUU}.)

Define $\Ss_{E} (\ww_{\PP^N})$ as the open subset of $\Stab (X)$ consisting of stability conditions $\sigma\in\Stab(X)$ of the form $G\cdot\widetilde{\sigma}$, for $G\in G_E$ and $\widetilde{\sigma}\in\Gg^E$.
Using \cite[Prop.\ 4.10]{B3} and \cite[Cor.\ 3.20]{emolo}, it follows that $\Ss_{E} (\ww_{\PP^N})$ is connected. A further topological study of $\Ss_{E} (\ww_{\PP^N})$, using the description of $\Ss_E (\PP^N)$, is contained in \cite{MM}.

\begin{remark} Note that the unique open subset $\Ss (\ww_{\PP^2}) := \Ss_{\{\ko_{\PP^2}, \ko_{\PP^2} (1), \ko_{\PP^2} (2)\}} (\ww_{\PP^2})$ is a little larger then the space $\Stab^0 (|\ww_{\PP^2}|)$ of \cite[Prop.\ 2.4]{B5}. Essentially we include the action of $\glpiu$ and more degenerate stability conditions.
\end{remark}

Unfortunately, for $N>1$, Proposition \ref{prop:compPN} cannot be improved. Thus we can just prove that there is an open subset of $\Ss_E (\PP^N)$, consisting of ``non-degenerate'' stability conditions, which is isomorphic to a fundamental domain in $\Ss_E (\ww_{\PP^N})$ with respect to the action of the group $G_E$.

This picture is clearer for the case $N=1$, when both the two
spaces cover the full stability manifolds. Indeed, we have the
following result, which is a particular case of a more general
theorem due to Ishii, Uehara and Ueda \cite{IUU}. Another proof of
it may be found in \cite{Okada2}. Here we present a shorter proof
which first appeared in \cite{emolo2}. The proof of
simply-connectedness may be useful in more general situations (see
\cite[Cor.\ 4.12]{emolo1} and \cite{MM}). Moreover, it does not
rely on the structure of the group of autoequivalences, but only
on Remark \ref{rmk:quivery}.

\begin{thm}\label{thm:stab1}
$\Ss (\ww_{\PP^1}) := \Ss_{O} (\ww_{\PP^1})=\Stab (|\ww_{\PP^1}|)$. In particular, $\Stab (|\ww_{\PP^1}|)$ is connected and its open subset $\Gg:=\Gamma^O$ is a fundamental domain for the action of the autoequivalences group. Moreover, $\Stab (|\ww_{\PP^1}|)$ is simply-connected.
\end{thm}

\begin{proof} We start by proving that $\Stab(|\ww_{\PP^1}|)$ is connected.
Say $\ss \in \Stab (|\ww_{\PP^1}|)$. For our convenience we take
the abelian category $\kp([0,1))$ as heart of the stability
condition $\ss$ instead of $\kp((0,1])$. This is no problem since,
up to the action of $\glpiu$, these two categories are the same.

By \cite[Prop.\ 2.9]{HMS} we know that every stable factor of a
spherical object must be spherical. It follows, since
$K(|\ww_{\PP^1}|)\cong \ZZ^2$, that there are at least two
$\ss$-stable spherical objects (with independent classes in
$K$-theory). Now, \cite[Prop.\ 1.6]{IU} says that any spherical
object is the image, up to shifts, of $\mc{O}_C$ via an
autoequivalence $\Ff\in G:=G_{\{\ko_C,\ko_C(1)\}}$. Since $G$
preserves the connected component  $\Sigma(|\omega_{\PP^1}|)$, we
may as well assume that $\mc{O}_C$ is a stable spherical object.
Write $\ks$ for another spherical stable object. Since $\mc{O}_C$
and $\ks$ are $\ss$-stable we may assume that, acting with
$\glpiu$, $\ff_{\ss} (\mc{O}_C)=0$ and that, up to shift,
$\ks\in\kp([0,1))$ and $\Hom^i_{\Db_0(X)}(\mc{O}_C,\ks) \neq 0$
only if $i\in\{0, 1\}$.

Suppose we know that $\ks\cong\ko_C(a)$, for some integer $a>0$. To show that $\mc{O}_C (-1) [1] \in \mc{P} ([0, 1))$, consider the triangle
\begin{equation}\label{eqn:tria}
\ko_C^{\oplus a+1} \stackrel{\psi}{\lto} \ko_C(a) \lto \ko_C(-1)^{\oplus a}[1],
\end{equation}
induced by the analogous one in $\Db(\PP^1)$. We claim that $\psi$
is an injection. For suppose not. Then, one can express
$\nk(\psi)$ as the extension $\kg' \lto \nk(\psi)\lto \kg$, with
$\kg' \in \kp(0,1)$ and $\kg\in \kp(0)$. The inclusion $\kg'
\hookrightarrow \ko_C^{\oplus a+1}$ forces $\kg' = 0$ and this
means that $\nk(\psi)$ is semistable of phase $0$. As $\ko_C$ is
stable, $\nk(\psi)\cong \ko_C^{\oplus k}$, for some $k>0$. This
and the triangle above give the conclusion that
$\Hom_{\Db_0(X)}(\ko_C, \ko_C(-1)^{\oplus a}) \neq 0$, which is
absurd. Therefore, $\ko_C(-1)^{\oplus a}[1] \in \kp([0,1))$, so
that $\ko_C(-1)[1] \in \kp([0,1))$. Since $\langle \ko_C(-1)[1],
\ko_C \rangle$ is the heart of a $t$-structure, and we have proved
that $\langle \ko_C(-1)[1], \ko_C \rangle \subset \kp([0,1))$,
Lemma 2.3 in \cite{emolo} yields $\langle \ko_C(-1)[1], \ko_C
\rangle = \kp([0,1))$. So, what we have shown is that, up to
component preserving autoequivalences and the action of $\glpiu$,
$\ss \in \Sigma(|\omega_{\PP^1}|)$.

If $\ks\cong\ko_C(b)[1]$, for some integer $b<0$, we reason in a similar way by considering the triangle
\begin{equation*}\label{eqn:tria1}
\ko_C^{\oplus(-b-1)}\lto\ko_C(b)[1]\lto \ko_C(-1)^{\oplus(-b)}[1].
\end{equation*}

Therefore it remains to prove that either $\ks=\ko_C(d)$, for some integer $d>0$, or $\ks=\ko_C(d)[1]$, for some $d<0$. By \cite[Cor.\ 3.10]{IU},
\begin{equation}\label{eqn:con1}
\bigoplus_q \mc{H}^q (\ks) \cong \mc{O}_C (a)^{\oplus r} \oplus \mc{O}_C (a+1)^{\oplus s},
\end{equation}
for some $a \in \ZZ$, $s \geq 0, r > 0$. So let us analyze all the
different possibilities.

\bigskip

\noindent {\sc Case $a\geq 2$ and $a\leq -3$.} First suppose
$a\geq 2$. Since $\Hom^i_{\Db_0(X)}(\mc{O}_C,\ks) \neq 0$ only if
$i\in\{0, 1\}$, an easy check using the spectral sequences
\begin{align}\label{eqn:spectr}
^{I\!}{E_2^{p,q}}=&\Hom_{\Db_0(X)}^p(\ko_C, \kh^q(\ks)) \Longrightarrow \Hom_{\Db_0(X)}^{p+q}(\ko_C,\ks)\\
\label{eqn:spectr2}
^{I\!I\!}{E_2^{p,q}}=&\Hom_{\Db(C)}^p(\ko_C,\ko_C(b)\otimes\wedge^q\kn_C)
\Longrightarrow \Hom_{\Db_0(X)}^{p+q}(i_*\ko_C, i_*\ko_C(b)),
\end{align}
where $\kn_C$ is the normal bundle of $C$ and $b\in\ZZ$, shows
that $\ks\cong\mc{O}_C (a)$. The case $a \leq -3$ is dealt with
similarly, using the spectral sequence
$\Hom^p_{\Db_0(X)}(\kh^{-q}(\ks),\ko_C)
\Longrightarrow\Hom_{\Db_0(X)}^{p+q}(\ks,\ko_C)$.

\bigskip

\noindent {\sc Case $a=1,-2$.} Both cases are completely similar, so we explicitly deal just with $a=1$. Using \eqref{eqn:spectr}, \eqref{eqn:spectr2} and the fact that $\nk(d_2^{0,2})= \Hom_{\Db_0(X)}(\ko_C, \kh^2(\ks))$ in \eqref{eqn:spectr}, it is easy to see that $\kh^q(\ks) \neq 0$ only if $q\in\{0,1\}$. Again, analyzing these spectral sequences one sees that $\kh^1(\ks)$ cannot contain $\ko_C(2)$, so $\kh^1(\ks) = \ko_C(1)^{\oplus r_1}$ for some $r_1 \geq 0$. The exact triangle
$$
\ko_C(1)^{\oplus r_0} \oplus \ko_C(2)^{\oplus s} \lto \ks \lto \ko_C(1)^{\oplus r_1}[-1]
$$
together with the vanishing $\Hom^2_{\Db_0(X)}(\ko_C(1), \ko_C(2)) = 0$ imply that $\ko_C(2)^{\oplus s}$ is a direct factor of $\ks$. But as $\ks$ is stable, this means that $\ks= \ko_C(2)$, $s=1$ and $r=0$, which contradicts our choice that $r >0$. So, one must have $s=0$.   The long exact sequence arising from the application of the functor $\Hom_{\Db_0(X)}(\ks[1],-)$ to the triangle
$$
\ko_C(1)^{\oplus r_0}  \lto \ks \lto \ko_C(1)^{\oplus r_1}[-1]
$$
shows that $\Hom_{\Db_0(X)}(\ks[1], \ko_C(1)^{\oplus r_0}) \cong \Hom_{\Db_0(X)}(\ks[1], \ks)$. It follows that the second map of the triangle above must be zero, for otherwise we get the contradiction that  $\Hom^{-1}_{\Db_0(X)}(\ks,\ks) \neq 0$. Using the fact that the cohomology of $\ks$ is concentrated in degrees $0$ and $1$, we get that $r_0=1$ and $\ks = \ko_C(1)$.

\bigskip

\noindent {\sc Case $a=0,-1$.} Consider first $a=0$, i.e.\ $\bigoplus_q \mc{H}^q (\ks) \cong \mc{O}_C^{\oplus r} \oplus \mc{O}_C (1)^{\oplus s}$ and define $l(\ks) := r + s$. If $l(\ks)=1$, one is done as above, so assume $l(\ks)>1$. By Lemma 4.2 in \cite{IU}, we have that $l(T_{\ko_C}(\ks)) < l(\ks)$. Also, $T_{\ko_C}(\ko_C)=\ko_C[-1]$. We see therefore that $\ko_C$ is stable for the stability condition $\tt = T_{\ko_C}(\ss)$ and $\bigoplus_q \mc{H}^q (\ks') \cong\mc{O}_C(a')^{\oplus r'} \oplus \mc{O}_C (a'+1)^{\oplus s'}$,
with $\ks' = T_{\ko_C}(\ks)$ and $r' + s'< l(\ks)$.

If $a'\not\in\{0,-1\}$, then we conclude applying the previous cases. If $a'=0$, then proceed further repeating the same procedure and considering $\ks'':=T_{\ko_C}(\ks')$ with $l(\ks'')<l(\ks')$. If $\bigoplus_q \mc{H}^q (\ks') \cong \mc{O}_C(-1)^{\oplus r'} \oplus \mc{O}_C^{\oplus s'}$, take $\Psi := (T_{\ko_C(-1)}(-) \otimes \pi^*\ko_{\PP^1}(2)[-1])$. Then, $l(\Psi(\ks')) =
l(T_{\ko_C(-1)}(\ks')) < l(\ks')$, and $\Psi(\ko_C) = \ko_C$. Repeating the argument a finite number of times, we reduce either to the case $a\not\in\{0,-1\}$ or to the case $l(\ks)=1$. In both cases we conclude that either $\ks=\ko_C(d)$, for some integer $d>0$, or $\ks=\ko_C(d)[1]$, for some $d<0$.

\bigskip

It remains to prove that $\Stab (|\ww_{\PP^1}|)$ is simply-connected.
Consider the open subset $V_0 := \til{U}_0^{O}$, with $O := \grf{\ko_{\PP^1}, \ko_{\PP^1} (1)}$.
Then the natural map from $V_0$ to
$$\kc := \grf{(m_0, m_1, \ff_0, \ff_1) \in \RR^4 \, : \, m_i > 0, \ \ff_1 -1 < \ff_0 < \ff_1 +1},$$
sending $\ss \in V_0$ to $(|Z(\ko_C [1])|, |Z(\ko_C (1))|, \ff (\ko_C [1]), \ff (\ko_C (1)))$ is an homeomorphism.
Hence $V_0$ is contractible and is the union of the three contractible regions
\begin{align*}
V_{0,1} &:= \grf{x \in \kc \, : \, \ff_0 < \ff_1} \cong \glpiu\\
V_{0,2} &:= \grf{x \in \kc \, : \, \ff_0 > \ff_1} \cong \glpiu\\
V_{0,3} &:= \grf{x \in \kc \, : \, \ff_0 = \ff_1} \cong \CC \times \RR_{>0},\\
\end{align*}
where $V_{0,2} = r V_{0,1}$ and $r$ was defined in Remark \ref{rmk:quivery}. By the same remark and by the previous part of the proof, we know that
$$\Stab (|\ww_{\PP^1}|) = \bigcup_{l \in B_2} l V_0,$$
where $l V_0$ is the set of stability conditions whose heart is, up to the action of $\glpiu$, the quivery subcategory $l \langle \ko_C [1], \ko_C (1) \rangle$.
Notice that $l V_0 \cap V_0 \neq \emptyset$ if and only if $l \in \grf{\tt_0^m, \tt_1^m, r}$ for some integer $m \in \ZZ$. Moreover in such a case, for all $m \in \ZZ$,
\begin{align*}
\tt_1^m V_0 \cap V_0 &= V_{0,1}\\
\tt_0^m V_0 \cap V_0 &= V_{0,2}.
\end{align*}
In particular, observe that $l V_{0,3} \cap V_{0,3} = \emptyset$
if $l$ is in the subgroup of $B_2$ generated by $\tau_0$ and
$\tau_1$, a fact which will be used implicitly later in the proof.
Thus we have the following easy consequence of Seifert--Van Kampen
theorem: $[*]$ Let $V$ be the open subset of $\Stab
(|\ww_{\PP^1}|)$ consisting of stability conditions whose heart
is, up to the action of $\glpiu$, a fixed quivery subcategory of
$\Db_0 (|\ww_{\PP^1}|)$ and let $h \in L := \grf{\tau_0, \tau_1,
r}$. Then $V \cup h V$ and $V \cup h^{-1} V$ are connected and
simply-connected.

Now, fix a point $\ss_0 \in V_{0,3}$. Take a continuous loop $\aa
: [0,1] \to \Stab (|\ww_{\PP^1}|)$ with base point $\ss_0$. Assume
that there exists $t_0 \in [0,1)$ such that $\aa ([0,t_0)) \cont
V_0$ and $\aa (t_0) \notin V_0$. Then there exist $i \in
\grf{0,1}$ and $m \in \ZZ$ such that $\aa (t_0) \in V_1:= \tt_i^m
V_0$. Continuing further, there exists $t_1 \in (t_0, 1]$ such
that $\aa ([t_0, t_1)) \cont V_1$ and $\aa (t_1) \notin V_1$, and
as before one has  $j \in \grf{0,1}$ and $n \in \ZZ$ such that
$\aa (t_1) \in V_2:= \tt_j^n V_1$. By compactness, and since
$\Stab (|\ww_{\PP^1}|)$ is a manifold, we may assume that $V_k =
V_0$, after a finite number of steps. Hence there exists $l \in F
(L)$ such that $l V_0 = V_0$, where $F (L)$ is the free group
generated by the set $L$. By Remark \ref{rmk:quivery}, the action
of $B_2$ on the set of ordered quivery subcategories is free.
Hence, up to multiplying by $r$, the class of $l$ in $B_2$ is
equal to the identity. This means that, up to contracting/adding
pieces of the form $h h^{-1}$ or $h^{-1}h$, with $h \in L$, by
$[*]$, we can assume
$$l = (k_1 s_1^{\pm 1} k_1^{-1}) \ldots (k_a s_a^{\pm 1} k_a^{-1}),$$
with $k_1, \ldots, k_a \in F(L)$ and $s_1, \ldots, s_a \in R(L)$, where $R(L) := \grf{r\tau_1r\tau_0^{-1}, \, r^2}$.
But then, again by $[*]$, the loop $\aa$ can be split as a composition of contractible loops, i.e.\ it is contractible.
Hence $\Stab (|\ww_{\PP^1}|)$ is simply-connected.\end{proof}

It seems reasonable that, in the case $N=2$, by adding all ``geometric'' stability conditions constructed along the lines of Example \ref{ex:K3ab}, we may be able to describe an actual connected component of both $\Stab (\PP^2)$ and $\Stab (|\ww_{\PP^2}|)$. In such a case, we would have again a fundamental domain in (a connected component of) $\Stab (|\ww_{\PP^2}|)$ isomorphic to an open ``non-degenerate'' subset of (a connected component of) $\Stab (\PP^2)$.

\begin{remark} (i) A statement analogous to Theorem \ref{thm:stab1} in the context of the local $3$-dimensional flop $\ko_{\PP^1}(-1)\oplus\ko_{\PP^1}(-1)$ over $\PP^1$ is a particular case of \cite[Thm.\ 1.1]{Toda1} and \cite[Thm.\ 7.2]{Toda2}.

(ii) We would like to mention that the case of $\Db_0(X)$ fits in with the equivariant examples considered in Section \ref{subsec:gen}. Indeed, as pointed out by Tony Pantev, the category $\Db_0 (\coh(X))$ of coherent sheaves on $X$ supported on the zero-section is equivalent to the derived category of $\omega_{\PP^N}^{\vee}$-equivariant coherent sheaves $\Db_{\omega_{\PP^N}^{\vee}}(\coh(\PP^N))$, with $\omega_{\PP^N}^\vee$ acting trivially on $\PP^N$.
 The proof will appear in \cite{MM}. Notice that the forgetful functor $\pi_*$ that we consider in this case does not satisfy condition \eqref{eq:ind}. Hence it does not induce stability conditions as in Section \ref{subsec:stabex}.\end{remark}

\bigskip

\subsection{Further examples}\label{subsec:moreex}~

\medskip

We conclude the paper with a brief discussion of a few examples related to the case considered in this section and where our techniques can be applied.

\subsubsection{Kleinian singularities} Let $\pi:X\to\CC^2/G$ be the minimal resolution of singularities, where $G$ is a finite subgroup of $\mathrm{Sl}_2(\CC)$, and $\cat{D}$ the full subcategory of $\Db(\coh(X))$ consisting of those objects $\ke$ such that $\pi_*\ke=0$. Stability conditions on $\cat{D}$ have been studied by Bridgeland in \cite{B4} (see also \cite{T} for the special case of $A_n$-singularities). More specifically, he describes a connected component $\Sigma$ which is a covering space of some period domain (see \cite[Thm.\ 1.1]{B4} for more details).

As mentioned in \cite{B4}, stability conditions on the category $\cat{D}$ should be related to those on the derived category of the path algebra of the Dynkin quiver associated to the group $G$. We can use our techniques to make this assertion precise. For simplicity we will explicitly treat only the case of $A_2$-singularity although the approach works in general.

Consider the Dynkin quiver $A_2:\,\bullet\to\bullet$ and the derived category $\Db(A_2)$ of its path algebra. Notice that, up to shifts, the category  $\Db(A_2)$ contains just three exceptional objects, two of which $\ks_0[1]$ and $\ks_1$ are the simples corresponding to the vertexes of $A_2$ and the third $\ks_2$ is the unique indecomposable extension of $\ks_0[1]$ and $\ks_1$. Stability conditions on $\Db(A_2)$ were described in \cite{emolo}.

Via the McKay correspondence, the category $\Db(X)$ is equivalent to $\Db(\widehat{A}_2)$, the derived category of the preprojective algebra associated to the affine Dynkin diagram $\widehat A_2$:
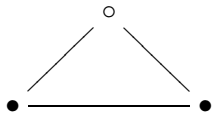
\[
\xymatrix{&\circ\ar@{-}[dr]\ar@{-}[dl]&\\\bullet\ar@{-}[rr]&&\bullet}
\]
By \cite{B4}, the natural faithful functor $\Db(A_2)\lto\Db(\widehat A_2)$ factorizes as $\Db(A_2)\mor[i]\cat{D}\mono\Db(\widehat A_2)$. Hence, the procedure described in Section \ref{subsec:stabex} may be applied to relate $\Stab(\Db(A_2))$ and $\Stab(\cat{D})$.

Consider $S_0:=\{\ks_0,\ks_1\}$, $S_1:=\{\ks_1,\ks_2\}$ and $S_2:=\{\ks_2[-1],\ks_0\}$ and the corresponding abelian categories $\cat{Q}^{S_i}_0$ defined as in \eqref{eqn:Q}. As in the previous section, it is easy to see that $\cat{Q}^{S_i}_0$ is $i$-admissible. Take the open subset $U^{S_i}\subseteq\Stab(\Db(A_2))$ consisting of stability conditions whose heart is $\cat{Q}^{S_i}_0$ (up to the action of $\glpiu$) and let $\Gamma=U^{S_0}\cup U^{S_1}\cup U^{S_2}$.

By Proposition \ref{pro:existence}, there exists an open connected subset $\widetilde\Gamma\subset\Stab(\cat{D})$ such that $i^{-1}:\widetilde\Gamma\to\Stab(\Db(A_2))$ is an open embedding whose image is $\Gamma$. By \cite[Thm.\ 1.1]{B4}, $\widetilde\Gamma$ is a fundamental domain for the action of the group $\Br(\cat{D})\subset\Aut(\cat{D})$ preserving the connected component $\Sigma$.


\subsubsection{Rational curves in K3 surfaces} Let $X$ be a smooth projective K3 surface containing a smooth rational curve $C$. Since the formal neighbourhood of $C$ in $X$ is isomorphic to the formal neighbourhood of $\PP^1$ in $\omega_{\PP^1}$, the category $\Db_0(|\omega_{\PP^1}|)$ can be seen as a fully faithful subcategory of $\Db(X)$. The embedding functor $i:\Db_0(|\omega_{\PP^1}|)\to\Db(X)$ can be used to induce stability conditions from $\Db(X)$ to $\Db_0(|\omega_{\PP^1}|)$.

Using the exact sequence $0\to\ko_X(-C)\to\ko_X\to\ko_C\to 0$, it is easy to see that the sheaves $\ko_C(k)$, $k\in\ZZ$, are not all stable in any stability condition $\sigma$ in the open subset $U(X)$ described in \cite[Sect.\ 10]{B2}. By \cite[Thm.\ 12.1]{B2}, there exists a stability condition $\sigma$ in the boundary $\partial U(X)_{(C_k)}$ of $U(X)$ of type $(C_k)$ (i.e.\ in $\sigma$ the objects $\ko_C(k)[1]$ and $\ko_C(k+1)$ are both stable of the same phase). Hence $\sigma$ induces a stability condition on $\Stab(\Db_0(|\omega_{\PP^1}|))$ via $i^{-1}$ (in other words, $\sigma\in\mathrm{Dom}(i^{-1})$). Since all stability conditions in $\NStab^\dag(\Db(X))$ are obtained from the stability conditions in the closure of $U(X)$ by applying autoequivalences, it is easy to see that $\mathrm{Dom}(i^{-1})\cap\NStab^\dag(\Db(X))$ concides with $$\{\Phi(\sigma)\in\NStab^\dag(\Db(X)):\sigma\in\partial U(X)_{(C_k)}\mbox{ and }\Phi\in\langle T_{\ko_C(m)}:m\in\ZZ\rangle\}.$$

Intuitively, letting the ample class $\omega$ in Example \ref{ex:K3ab} go to infinity produces a degeneration of Bridgeland's stability conditions which, by acting with the group $\langle T_{\ko_C(m)}:m\in\ZZ\rangle$, yields a surjection onto $\Stab(\Db_0(|\omega_{\PP^1}|))$.

\medskip

{\small\noindent{\bf Acknowledgements.}  The final version of this
paper was written during the authors' stay at the
Max-Planck-Institut f\"{u}r Mathematik (Bonn) whose hospitality is
greatfully acknowledged. We thank Tom Bridgeland and Daniel
Huybrechts for useful conversations and for going through a
preliminary version of this paper. The first and second named
authors are grateful to Tony Pantev for useful discussions. S.M.\
thanks Daniel Huybrechts for inviting him to the University of
Bonn; he is especially grateful to Tom Bridgeland and the
University of Sheffield for support and hospitality for the
academic year 2006--07. E.M.\ expresses his gratitude to Ugo
Bruzzo for advice and many suggestions at a preliminary stage of
this work. We also thank Yukinobu Toda for useful conversations
and for pointing out the reference \cite{Pol}. Input from So Okada
for a different project involving a similar circle of ideas is
appreciated. Finally, we want to thank the referee for carefully
reading the paper and for pointing out several improvements.}


\end{document}